\documentclass[amsart,11pt]{article}
\usepackage{mathtools}
\usepackage{color,graphicx,latexsym,amssymb,amsmath,amsthm,amsfonts,float}
\usepackage{url}
\usepackage[numeric,initials,nobysame]{amsrefs}
\usepackage{hyperref,mathrsfs}
\usepackage{hyperref}
\hypersetup{colorlinks}
\usepackage{authblk}
\usepackage{enumitem}
\newtheorem{theorem}{Theorem}
\newtheorem{lemma}{Lemma}

\newtheorem{prop}{Property}

\theoremstyle{definition}
\newtheorem{remark}{Remark}
\newtheorem{definition}{Definition}
\newcommand{\E}{{\mathbb E}}
\newcommand{\e}{{\epsilon}}
\newcommand{\PP}{{\mathbb{P}}}

\newcommand{\dd}{{\delta}}

\newcommand{\symdiff}{{\vartriangle}}
\newcommand\moniker[1]{{\em (#1)}}
\newcommand\old[1]{}
\title{Formation of an interface by competitive erosion}
\author[1]{Shirshendu Ganguly\thanks{sganguly@math.washington.edu}}
\author[2]{Lionel Levine\thanks{\url{www.math.cornell.edu/\~levine}. 
 Supported by NSF grant DMS-1243606 and a Sloan Fellowship.}}
\author[3]{Yuval Peres\thanks{peres@microsoft.com}} 
\author[4]{James Propp\thanks{\url{www.jamespropp.org}. 
 Partially supported by NSF grant DMS-1001905.}}
\affil[1]{University of Washington}
 \affil[2]{Cornell University}
\affil[3]{Microsoft Research}
\affil[4]{University of Massachusetts Lowell}
\def\R{\mathbb{R}}
\def\Z{\mathbb{Z}}
\def\div{\mathop{\mathrm{div}}}

\def\Cyl{\mathrm{Cyl}}

\begin{document}
\maketitle

\begin{abstract}
\noindent
In 2006, the fourth author of this paper
proposed a graph-theoretic model of interface dynamics 
called {\bf competitive erosion}.
Each vertex of the graph is occupied by a particle that can be either red or blue.  New red and blue particles alternately get emitted from their respective  bases and perform random walk. On encountering a particle of the opposite color they kill it and occupy its position. 
We prove that on the cylinder graph (the product of a path and a cycle) an interface spontaneously forms between red and blue and is maintained in a predictable position with high probability.
\end{abstract}


\begin{figure}[here]
\centering
\begin{tabular}{ccc}
\includegraphics[width=.3\textwidth]{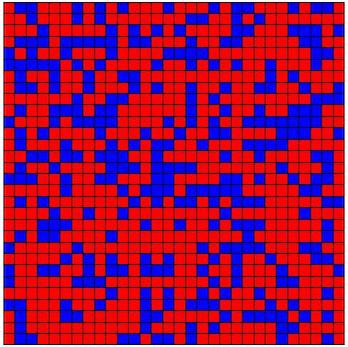} &
\includegraphics[width=.3\textwidth]{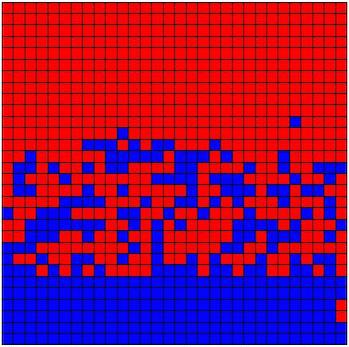} &
\includegraphics[width=.3\textwidth]{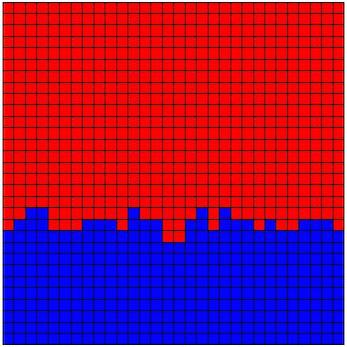} \\
$t=0$ & $t=125$ & $t=325$
\end{tabular}
\caption{ Competitive erosion spontaneously forms an interface.
In the initial state $S(0)$ each vertex of the $30\times 30$ cylinder graph is independently colored blue with probability $.33$ and red otherwise. After 325 time steps of competitive erosion, an interface has formed between red and blue.
\label{f.formation}
}
\end{figure}

\section{An Interface In Equilibrium}

We introduce a graph-theoretic model of a random interface maintained in equilibrium by equal and opposing forces on each side of the interface. Our model can also be started from a heterogeneous state with no interface, in which case an interface forms spontaneously. 
Here are the data underlying our model, 
which we call \emph{competitive erosion}:
	\begin{itemize}
	\item a finite connected graph on vertex set $V$;
	\item probability measures $\mu_1$ and $\mu_2$ on $V$; and
	\item an integer $0 \leq k \leq \#V-1$.
	\end{itemize}
\emph{Competitive erosion} is a discrete-time Markov chain $(S(t))_{t \geq 0}$ on the space of all subsets of $V$ of size $k$. One time step is defined by
	\begin{equation} \label{e.onestep} S(t+1) = (S(t) \cup \{X_t\}) - \{Y_t\} \end{equation}
where $X_t$ is the first site in $S(t)^c$ visited by a simple random walk whose starting point has distribution $\mu_1$, and $Y_t$ is the first site in $S(t) \cup \{X_t\}$ visited by an independent simple random walk whose starting point has distribution $\mu_2$.

For a concrete metaphor, one can imagine that $S(t)$ and its complement represent the territories of two competing species. We will call the vertices in $S(t)$ ``blue'' and those in $S(t)^c$ ``red''.  According to \eqref{e.onestep}, a blue individual beginning at a random vertex with distribution $\mu_1$ wanders until it encounters a red vertex $X_t$ and inhabits it, evicting the former inhabitant. The latter returns to an independent random vertex with distribution $\mu_2$ and from there wanders until it encounters a blue vertex $Y_t$ and inhabits it, evicting the former inhabitant.

If the distributions $\mu_1$ and $\mu_2$ have well-separated supports, one expects that these dynamics resolve the graph into coherent red and blue territories separated by an interface that, although microscopically rough, appears in a macroscopically predictable position with high probability.  The purpose of this article is to prove this assertion in the case of the {\em cylinder graph} 
	$ \mathrm{Cyl}_n = C_n \times P_{n}, $
where $C_n$ is a cycle of length $n$ and $P_n$ is a path of length $n$. 
Here $\times$ denotes the Cartesian (box) product of graphs.
We identify $C_n$ with $(\frac1n \Z)/\Z$ and $P_n$ with $(\frac1n \Z) \cap [0,1]$.  We take $\mu_1$ and $\mu_2$ to be the uniform distributions on $C_n \times \{0\}$ and $C_n \times \{1\}$, respectively: blue particles are released from the base of the cylinder and red particles from the top.

Taking $k=\alpha n^2$ for some $0<\alpha<1$, it is natural to guess that the stationary distribution of the Markov chain $S(t)$ assigns high probability to the event 
	\begin{equation} \label{goodset11} A_{\epsilon,n} := \{ S \subset \mathrm{Cyl}_n \,:\, C_n \times [0, \alpha-\epsilon] \subset S \subset C_n \times [0,\alpha+\epsilon] \}, \end{equation}
that is, the event that sites below the line $y=\alpha-\epsilon$ are all blue and that sites above the line $y=\alpha+\epsilon$ are all red.  Our main result confirms this guess.

\begin{theorem}\label{mainresult}
Given $\e>0$ there exists a positive constant $d=d(\e)$ such that      
\begin{equation}\label{band}
\pi_{n}(\mathcal{A}_{\e,n})\ge 1-e^{-dn}
\end{equation} 
where $\pi_{n}$ is the stationary distribution 
of the competitive erosion chain on $\mathrm{Cyl}_n$.  
\end{theorem}

The proof of Theorem~\ref{mainresult} will show that the predicted interface forms quickly, even if the initial state $S(0)$ is heterogeneous (such as a checkerboard of red and blue, or a random initial state like the one shown in Figure~\ref{f.formation}). 

\begin{theorem}
\label{t.quickhit}
Let $\tau_\e = \inf \{t \,:\, S(t) \in \mathcal{A}_{\epsilon,n} \}$. For any $\e>0$ there exist
positive constants $c=c(\e),d=d(\e),N=N(\e)$ such that for all $n>N$ and all subsets $S \subset \Cyl_n$ of cardinality $\lfloor \alpha n^2 \rfloor$ we have
\[
\mathbb{P}(\tau_\e >dn^2 \ | \ S(0)=S)  < e^{-cn}.
\]
\end{theorem}
See the last section for some related models exhibiting interface formation. We also mention a superficially similar process 
called ``oil and water'' \cite{oilwater} which has an entirely different behavior:
the two species do not form a macroscopic interface at all.

\subsection{Comparison with IDLA}\label{idla1}

Internal diffusion limited aggregation (IDLA) is a fundamental model of a 
random interface moving in a \emph{monotone} (outward) fashion. 
IDLA involves only one species with an ever-growing territory 
	\[ I(t+1) = I(t) \cup \{X_t\} \]
where $X_t$ is the first site in $I(t)^c$ visited by a simple random walk whose starting point has distribution $\mu_1$. 
Competitive erosion can be viewed as a symmetrized version of IDLA: whereas $I(t)$ and $I(t)^c$ play asymmetric roles, $S(t)$ and $S(t)^c$ play symmetric roles in \eqref{e.onestep}.

IDLA on a finite graph is only defined up to the finite time $t$ when $I(t)$ is the entire vertex set. For this reason, the IDLA is usually studied on an infinite graph and the theorems about IDLA are limit theorems: asymptotic shape \cite{lbg}, order of the fluctuations \cites{ag,ag2,jls}, and distributional limit of the fluctuations \cite{gff}.  In contrast, competitive erosion on a finite graph is defined for all times, so it is natural to ask about its stationary distribution.  
To appreciate the difference in character between IDLA and competitive erosion, note that
the stationary distribution of the latter assigns tiny but positive probability to configurations that look nothing like the predicted horizontal interface. 
Competitive erosion will occasionally form these exceptional configurations: for example, at a tiny but positive fraction of times $t$ the boundary of the set $S(t)$ is a \emph{vertical} line! 
The proof of Theorem~\ref{mainresult} is delicate 
since there are so many exceptional configurations: 
in terms of cardinality, the desired set $\mathcal{A}_{\e,n}$ is 
an exponentialy small fraction of the set of all recurrent configurations.
\subsection{Idea of the proof}\label{pidea}
To prove that the stationary distribution concentrates on the small set $\mathcal{A} = \mathcal{A}_{\e,n}$, we identify a \emph{Lyapunov function} $h$ on the state space which attains its global maximum in $\mathcal{A}$ and increases in expectation 
	\begin{equation} \label{e.thedrift1} \E (h(S(1)) - h(S(0))) \geq a >0 \end{equation}
provided $S(0)$ is sufficiently far from $\mathcal{A}$.  The function $h$ is as simple as one could hope for: the sum of the heights of the red vertices. The heart of the proof is Theorem~\ref{nonnegativedrift}, which uses an electrical resistance argument to establish the drift \eqref{e.thedrift1} for a suitable notion of ``sufficiently far from $\mathcal{A}$''.   Since the function $h(\sigma)$ is bounded by $n^2$, we then use Azuma's inequality to argue that the process $h(S(t))$ spends nearly all its time in a neighborhood of its maximum.
This establishes Theorem \ref{thmdust} (a statistical version of Theorem \ref{mainresult}). 

%
 The main remaining difficulty lies in showing that if $S(0)$ is sufficiently \emph{close} to $\mathcal{A}$, then the chain $S(t)$ hits $\mathcal{A}$ quickly with high probability, establishing Theorem \ref{t.quickhit}. This is done using stochastic domination arguments involving IDLA on the cylinder.
These ingredients together with a general estimate relating hitting times to stationary distributions (Lemma~\ref{hitstation}) establish Theorem~\ref{mainresult}.

\subsection{The level set heuristic}
\label{s.levelset}

Before restricting to the cylinder we mention a heuristic that predicts the location of the competitive erosion interface for well-separated measures $\mu_1,\mu_2$ on a general finite connected graph, which we assume for simplicity to be $r$-regular.
 Let $g$ be a function on the vertices satisfying
	\begin{equation} \label{e.generalg} \Delta g = \mu_1 - \mu_2 \end{equation}
where $\Delta$ denotes the Laplacian 
	\[ \Delta g(x) :=  \frac{1}{r} \sum_{y \sim x} (g(x)-g(y)) \]
and the sum is over vertices $y$ neighboring $x$.  Since the graph is assumed connected, the kernel of $\Delta$ is one-dimensional consisting of the constant functions, so that equation \eqref{e.generalg} determines $g$ up to an additive constant.  

The \emph{boundary} of a set of vertices $S \subset V$ is the set
	\begin{equation} \label{e.boundary} \partial S = \{v \in V-S \,:\, v\sim w \text{ for some } w \in S\}. \end{equation}
Consider a partition of the vertex set $V = S_1 \sqcup B \sqcup S_2$ where $\partial S_1 = B = \partial S_2$. (Think of $S_1$ and $S_2$ as the blue and red territories respectively, of the sort we might expect to see in equilibrium,
and the sites in their common boundary $B$ have indeterminate color.)

Let $g_i$ for $i=1,2$ be the Green function for random walk started according to $\mu_i$ and stopped on exiting $S_i$. These functions satisfy
	\begin{eqnarray*} \Delta g_i &= & \mu_i \quad\text{on } S_i , \\
			  g_i &= & 0  \quad\,\, \text{on } S_i^c.
			  \end{eqnarray*}
The probability that simple random walk started according to $\mu_i$ first exits $S_i$ at $x \in B$ is $-\Delta g_i(x)$.

To maintain equilibrium in competitive erosion, we seek a partition such that $\Delta g_1 \approx \Delta g_2$ on $B$, that is  \[ \Delta (g_1 - g_2) \approx \mu_1 - \mu_2. \]  (Exact equality holds except on $B$.)
Thus by \eqref{e.generalg}, the function $g - (g_1 - g_2)$ is approximately constant.  Since $g_i$ vanishes on $B$, the equilibrium interface $B$ should have the property that
	\[ \text{$g$ is approximately constant on $B$}. \]
The partition that comes closest to achieving this goal takes $S_1$ to be the level set 
	\begin{equation} \label{e.levelset} S_1 = \{x \, :\, g(x) < K \} \end{equation}
for a cutoff $K$ chosen to make $\# S_1 = k$.  An application of the maximum principle shows that for this choice of $S_1$, the maximum and minimum values of $g - (g_1-g_2)$ differ by at most
	\[ \max_{x \in S_1, y \notin S_1, x \sim y} |g(x)-g(y)| , \]
suggesting that the right notion of ``well-separated'' measures $\mu_1$ and $\mu_2$ is that the resulting function $g$ has small gradient.
\old{
Then $g-g_1$ is harmonic on $B_1$ and $g+g_2$ is harmonic on $B_2$. By the maximum principle,
	\begin{align*} m_1 \leq g - g_1 \leq M_1 \qquad \text{on } S_1 \end{align*}
and
	\begin{align*} m_2 \leq g + g_2 \leq M_2 \qquad \text{on } S_2 \end{align*}
where $m_i$ and $M_i$ are respectively the minimum and maximum values of $g$ on the boundary $\partial B_i$. Therefore, since the supports of $g_1$ and $g_2$ are disjoint, there is a constant $C$ such that
	\[ | g-g_1+g_2-C | \leq \frac12 \max(M_1-m_1, M_2-m_2). \]
If we take $B_1$ to be a level set of $g$, then for $i=1,2$ we have
	\[ M_i - m_i \leq \max_{x \in \partial B_1, \, y\in \partial B_2, \, x\sim y} |g(x)-g(y)| \]
(take $x \in \partial B_1$ attaining the maximum $M_1$ and a neighbor $y \in \partial B_2$ of $x$. Then $x \in B_2$ and $y \in B_1$; since $B_1$ is a level set we have $g(y) \leq \min_{z \in B_2} g(z) \leq m_1$ else $y$ would be in $B_2$).
}
\subsubsection{Mutually annihilating fluids}

The divisible sandpile \cite{div1,div2} is a deterministic analogue of IDLA.
In a competitive version of the divisible sandpile, 
the edges of our graph form a network of pipes containing a red fluid and a blue fluid. 
The two fluids are injected at respective rates $\mu_1$ and $\mu_2$ and annihilate upon contact. One can convert the above heuristic into a proof that this deterministic model has its interface exactly at a level line of $g$ (where $g$ is interpolated linearly on edges).

\subsection{Diffusive sorting}

We define here a process called \emph{diffusive sorting},
introduced by the fourth author,
which couples the competitive erosion chains on a given graph for all values of $k$, using a single random walk to drive them all. We will not use the coupling in this paper, but find it to be of intrinsic interest.

It is convenient to imagine that the blue and red random walks start from vertices $v_1$ and $v_2$ respectively that are external to the finite connected graph $G$ and have directed weighted edges into $G$: each edge $(v_i,v)$ has weight $\mu_i(v)$.

The state space of diffusive sorting consists of bijective label lings $\lambda$ of the vertex set of $G$ by the integers $\{1,\ldots,N\}$, where $N$ is the number of vertices.
We imagine that the labels are detachable from the vertices
and can be carried temporarily by the random walker.
Consider an initial labeling $\lambda_0$.
A random walker starts at $v_1$ carrying the label~$0$.
Whenever it comes to a vertex whose label exceeds the label the walker is currently carrying,
the walker swaps labels with that vertex,
dropping its old (smaller) label and stealing the new (larger) label for itself.
Eventually the walker will visit the vertex labeled $N$ and acquire the label~$N$ for itself, 
at which point we may stop the walk, since the particle will carry the label $N$ forever after and no labels will change.
The vertex labels are now $\{0,...,N-1\}$. We now increase all those labels by 1, and call the result $\lambda_{1/2}$.
 (We have noticed that this process bears a strong resemblance to an algorithm in algebraic combinatorics called promotion \cite{Stanley}, but we have not explored this connection.)
To get from $\lambda_{1/2}$ to $\lambda_1$, we apply the same process from the other side,
releasing a random walker from $v_2$ bearing the label $N+1$,
and letting it swap labels with any vertex it encounters whose label is smaller than its own,
until it acquires the label $1$; then we decrease the labels of all vertices by 1,
obtaining $\lambda_1$.

\emph{Diffusive sorting} is the Markov chain $(\lambda_t)_{t \geq 0}$ on labelings, each of whose transitions from $\lambda_t$ to $\lambda_{t+1}$ is as described in the previous paragraph. To see how this chain relates to competitive erosion, for each integer $k=0,\ldots,N-1$ let 
	\[ S_k(t) = \{ v \,:\, \lambda_t(v) \leq k \}. \]
Then it is not hard to check that $S_k$ has the law of the competitive erosion chain. (The key observation is that the vertices labeled $1$ through $k$ in $\lambda_0$
must be a subset of the vertices labeled $1$ through $k+1$ in $\lambda_{1/2}$, and similarly for going from $\lambda_{1/2}$ to $\lambda_1$.)

The following result shows that the stationary distribution of diffusive sorting on the cylinder graph is concentrated on labelings that are uniformly close to the function $n^2y$. 

\begin{theorem}\label{diffthm1}
For each $\e>0$ there is a constant $d=d(\e)>0$ such that
	\[ \hat \pi_n \left\{ \lambda \,:\, \sup_{(x,y) \in \mathrm{Cyl}_n}\left | \frac{\lambda(x,y)}{n^2} - y \right| > \e \right\} < e^{-dn}. \]
where $\hat \pi_n$ is the stationary distribution of the diffusive sorting chain $(\lambda_t)_{t \geq 0}$ on $\mathrm{Cyl}_n$.
\end{theorem}

The level set heuristic (\textsection\ref{s.levelset}) makes a prediction for diffusive sorting on a general graph: if the gradient of $g$ is not too large and we label the vertices by $\{g(v) \,:\, v\in V\}$ instead of by $\{1,\ldots,N\}$, then $\hat \pi_n$ should concentrate on labelings not too far from the function $g$.

We mention, as an aside, that if instead of alternating between
blue walkers and red walkers we use walkers of just one color,
we get a sorting version of ordinary IDLA. For any $k \geq 1$ and $t \geq 0$ the law of $S_k(t+k)$ is that of the IDLA cluster with $k$ particles. (This is trivial for $k=1$ and the rest follows by induction on $k$.) 

\section{Connectivity properties of competitive erosion dynamics}\label{cpe}

Implicit in the statement of Theorem~\ref{mainresult} is the claim that competitive erosion has a unique stationary distribution.  In this section we formally define the competitive erosion chain and prove this claim.

\subsection{Formal definition of competitive erosion}\label{model}


The graph $\Cyl_n$ is a discretization of the cylinder $(\R/\Z) \times [0,1]$ with mesh size $\frac{1}{n}.$
To avoid degeneracies, we will always assume $n \geq 2$.
Let $P_n = (\frac1n \Z) \cap [0,1]$ be the path graph with edges $\frac{k}{n} \sim \frac{k+1}{n}$ for $k=0,1,\ldots,n-1$.
Let $C_n = (\frac{1}{n}\Z)/\Z$ be the cycle graph obtained by gluing together the endpoints $0$ and $1$ of $P_n$. Let
\begin{equation}\label{graphrep}
\Cyl_n= C_n \times P_n
\end{equation} 
with edges $(x,y) \sim (x',y')$ if either $x=x'$ and $y \sim y'$, or $x\sim x'$ and $y=y'$. We also add a self-loop at each point in $C_n \times \{0\}$ and $C_n \times \{1\}$, so that $\Cyl_n$ is a $4$-regular graph. 
We will use $\Cyl_n$ to denote both this graph and its set of vertices.

For $t=0,1,2,\ldots$ let $(X^{(t)}_s)_{s \geq 0}$ and $(Y^{(t+\frac12)}_s)_{s \geq 0}$ be independent simple random walks in $\Cyl_n$ with
	\[ P(X^{(t)}_0 = (x,0)) = \frac1n = P(Y^{(t+\frac12)}_0 = (x,1)) \]
for all $x \in C_n$. That is, each walk $X^{(t)}$ starts uniformly on $C_n \times \{0\}$ and each walk $Y^{(t+\frac12)}$ starts uniformly on $C_n \times \{1\}$.  Given the state $S(t)$ of the competitive erosion chain at time $t$, we build the next state $S(t+1)$ in two steps as follows.
	\[ S(t+\frac12) = S(t) \cup \{X^{(t)}_{\tau(t)}\} \]
where $\tau(t) = \inf \{s \geq 0 \,:\, X^{(t)}_s \not \in S(t)\}$. Let
	\[ S(t+1) = S(t+\frac12) - \{Y^{(t+\frac12)}_{\tau(t+\frac12)} \} \]
where $\tau(t+\frac12) = \inf \{s \geq 0 \,:\, Y^{(t+\frac12)}_s \in S(t+\frac12) \}$.
	
To highlight the symmetrical roles played by $S(t)$ and its complement, we will often think of the states as $2$-colorings rather than sets: let
	\begin{equation} \label{e.thecoloring} \sigma_t(x) = \begin{cases} 1 & x \in S(t) , \\ 
						2 & x \notin S(t). \end{cases} \end{equation}

Also define for $i=1,2$
\begin{equation}\label{region}
B_i:=B_{i}(\sigma):=\{x \in \Cyl_n: \sigma(x)=i\}.
\end{equation}
Thus $$S(t)=B_{1}(\sigma_{t}).$$
We will use color $1$ and ``blue'' interchangeably 
and similarly color $2$ and ``red''. 
Now for each $\alpha$ (the proportion of blue sites) strictly between 0 and 1,
competitive erosion defines a Markov chain on the space 
\begin{equation}\label{statespace}
\Omega:=\{\sigma\in \{1,2\}^{\Cyl_n}, 
\#\{x \in \Cyl_n : \sigma(x)=1\}=\lfloor{\alpha|\Cyl_n|}\rfloor\}.
\end{equation}
We also set
\begin{equation}\label{statespace1}
\Omega':=\{\sigma\in \{1,2\}^{\Cyl_n}, 
\#\{x \in \Cyl_n : \sigma(x)=1\}=\lfloor{\alpha|\Cyl_n|}\rfloor+1\}.
\end{equation}
A single time step of competitive erosion consists of a step from $\Omega$ to $\Omega'$ followed by a step from $\Omega'$ back to $\Omega$.

\subsection{Notational conventions}\label{nc}
We will often use the same letter 
(generally $C$, $D$, $c$ or $d$)  
for a constant whose value may change from line to line 
(or sometimes even within one line);
this convention obviates the need for distracting subscripts
and should cause no confusion.

For any process and a subset $A$ of the corresponding state space 
$\tau(A)$ will denote the hitting time of that set.
When $\omega$ is a state in the space and $A$ is a set of states,
we will write $\{\tau(A) \leq K \mid{\omega}\}$ to denote
the event that, starting from $\omega$, 
the process hits $A$ in at most $K$ steps.
Finally, in all the notation the dependence on $n$ will be often suppressed.

\subsection{Blocking sets and transient states}

\begin{definition}
Call a subset $A\subset \Cyl_n$  {\bf blocking}
if $\Cyl_n\setminus A$ is disconnected and the subsets 
$C_n \times \{0\}\setminus A $ and $C_n \times \{1\}\setminus A$
lie in different components.
\end{definition}
\noindent
We next prove that the competitive erosion chain has exactly 
one irreducible class and hence has a well-defined stationary measure. 
We start with a definition. 
\begin{definition}For two disjoint blocking subsets $A,B \subset \Cyl_n$ 
we say that $A$ is {\bf over} $B$ if 
\begin{enumerate}
\item $A$ and $C_n \times \{0\}\setminus B $ 
lie in different components of $\Cyl_n\setminus B$, and
\item $B$ and $C_n \times \{1\}\setminus A $ 
lie in different components of $\Cyl_n\setminus A$.
\end{enumerate}
\end{definition}
 \begin{figure}[hbt] 
\centering
\includegraphics[scale=.6]{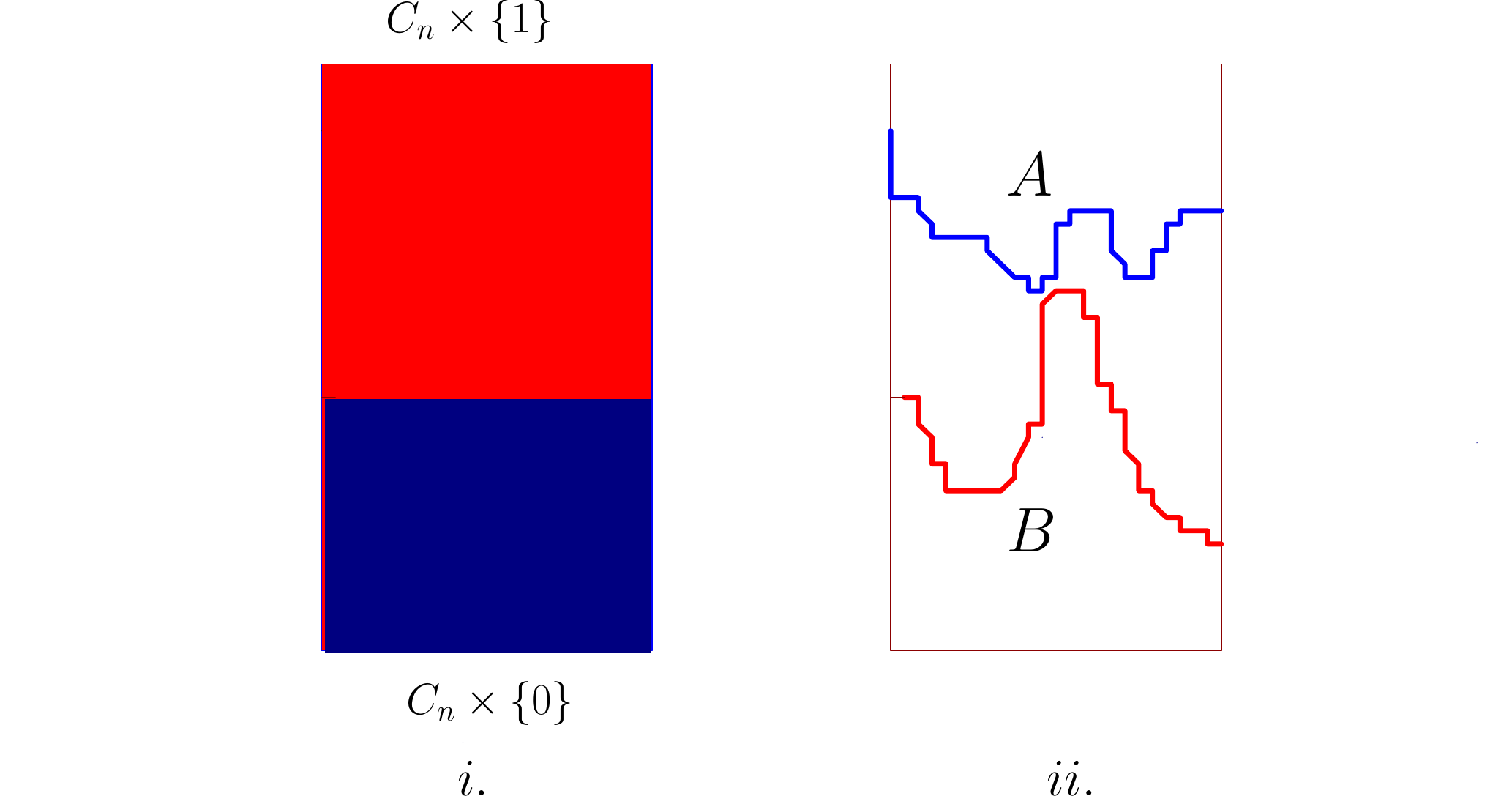}
\caption{ $i.$ The configuration $\sigma_{*}$. $ii.$ A configuration with a blue blocking subset over a red blocking subset.
\label{fig:blocking}}
\end{figure}

\begin{lemma}\label{irreducible} The competitive erosion chain has exactly one irreducible class. Moreover any $\sigma \in \Omega$ that has a blue blocking set over a red blocking set is transient. 

\end{lemma}
\begin{proof}
Consider the configuration $\sigma_*$ in Figure~\ref{fig:blocking} where the lowest $\lfloor \alpha n(n+1)\rfloor$ vertices are colored blue.  
To prove the first statement notice that from any $\sigma$ 
one can reach $\sigma_*$.
Since in the target configuration there is exactly one blue component, 
we look at the closest vertex with $\sigma$ value $2$ from $C_n \times \{0\}$.
Since there is a blue path from $C_n \times \{0\}$ to that point the Markov chain allows us 
to change it to $1$ and similarly at the other end. 
Thus we are done by repeating this. 

To prove the second statement we prove that starting from $\sigma_{0}$ 
one cannot reach any configuration like the one on the right 
that has one blue blocking set over another red blocking set.
We formally prove this by contradiction. 
Let $\sigma$ be a configuration 
with a blue blocking set $B$ over a red blocking set $A$. 
Now the competitive erosion chain evolves 
as $$\sigma_{0},\sigma_{1/2},\sigma_1,\sigma_{3/2}\ldots $$ 
where for every non-negative integer $k,$ $\sigma_k \in \Omega$ 
and $\sigma_{k+1/2}\in \Omega'.$
Assume now that there is a path 
$$\sigma_{0},\sigma_{1/2},\sigma_1,\sigma_{3/2}\ldots, \sigma_t=\sigma.$$
Let $\tau'$ and $\tau''$ be the last times along the path 
such that at least one vertex of $A$ is blue and similarly $B$ 
contains at least one red vertex respectively. 
If such a time does not exist let us call it $-\infty.$
Since $\sigma_0$ does not have a blue blocking set over a red blocking set, 
$$\max(\tau',\tau'')> -\infty.$$
$\tau'$ must be a half integer (since at half integers there is one more blue particle)
and similarly $\tau''$ must be an integer.
Thus $$\tau' \neq \tau''.$$
Next we see that we cannot have $\tau'<\tau''$. 
For, suppose otherwise.  Then for all times
greater than $\tau',$ $A$ is entirely red.
Thus no blue walker crosses $A$ at any time greater than $\tau'$. 
So $B$ must already be blue at $\tau'$ and stays blue through till $t$ 
implying that $\tau''< \tau'.$  
By a similar argument we cannot have $\tau'> \tau''.$  
Hence we arrive at a contradiction.
\end{proof}

\subsection{Organization of the proofs}\label{oa}

In Section \ref{pot2} we state a weak ``statistical" version of the main result, Theorem~\ref{thmdust}, and provide its proof  using hitting time estimates whose proof appear in Section \ref{phtr} . 
A key step in the proof, the construction of a suitable Lyapunov function, appears in Section \ref{secmtr}. This section uses the theory of electrical networks. A short review of a few basic facts about electrical networks that are assumed in this section  appears in Appendix \ref{app2}. 
In Section \ref{2} we deduce the stronger Theorem~\ref{mainresult}. As a simple corollary we obtain Theorem \ref{diffthm1}. This last section uses an estimate for IDLA on the cylinder proved in Appendix \ref{pf}. We also apply a hitting time estimate for submartingales several times throughout the paper. The proof of the estimate is included in Appendix \ref{newapp}. The final section discusses some related models and  future directions.

\section{Statistical version of the main theorem}\label{pot2}

Given $\e>0$ we define the set 
\begin{equation}\label{defdust}
\mathcal{G}_{\e}:=\left\{\sigma \in \Omega: |\sigma^{-1}(1) \symdiff \{y\le \alpha\}|\le \e n^2
\right\}.
\end{equation}
where $\symdiff$ denotes the symmetric difference of sets, 
$A \symdiff B := (A^c \cap B) \cup (A \cap B^c)$.
\begin{theorem}\label{thmdust}Given $\e>0$ there exists $c>0$ 
such that for large enough $n$
	\[ \pi_n (\mathcal{G}_{\e})\ge 1-e^{-cn}. \]
\end{theorem}

This theorem says that, 
after running the competitive erosion chain for a long time, 
one is unlikely to see many blue sites below the line $y=\alpha$ 
or many red sites above it, which 
(see Remark \ref{containment11} iii.\ below) 
is a weaker statement than Theorem \ref{mainresult}.

\subsection{The height function}
We now introduce a function on the state space 
that will appear throughout the rest of the article 
and will be used as a Lyapunov function,
along lines sketched in \eqref{e.thedrift1}. 
The interested reader can learn more about Lyapunov functions in \cite{fmm}.
Given $\sigma\in \Omega\cup \Omega'$ 
(recall~\ref{statespace} and~\ref{statespace1}),
let us define the height function of $\sigma$ as

\begin{equation}\label{wf}
h(\sigma)=\sum_{(x,y)\in B_1(\sigma)}(1-y)
\end{equation} where $B_{1}(\sigma)$ is defined in \eqref{region}.  
Clearly
\begin{equation}\label{maxbnd}
\sup_{\sigma\in \Omega}h(\sigma)=h(\sigma_{max})
=\alpha(1-\frac{\alpha}{2}) n^2+O(n),
\end{equation}
where $\sigma_{max}$ is any element of $\sigma$
with $\sigma(x,y) = 1$ for the lowest $\lfloor{\alpha n(n+1)\rfloor}$ vertices.
Also notice that for all $t$ 
\begin{equation}\label{gradbound}
|h(\sigma_{t})-h(\sigma_{t-1})|\le 1.
\end{equation}
\begin{definition}\label{wrongdef}
Throughout the rest of the article given $\sigma \in \Omega \cup \Omega'$ we say that a vertex $v=(x,y)$ has the \emph{wrong color} if either $y<\alpha$ and $\sigma(v)$ is red or if $y>\alpha$ and $\sigma(v)$ is blue. 
\end{definition}

\begin{figure}{} 
\centering
\begin{tabular}{ccc}

\includegraphics[width=.2\textwidth]{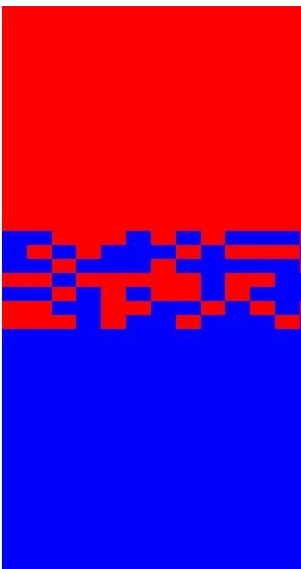} \quad &
\includegraphics[width=.2\textwidth]{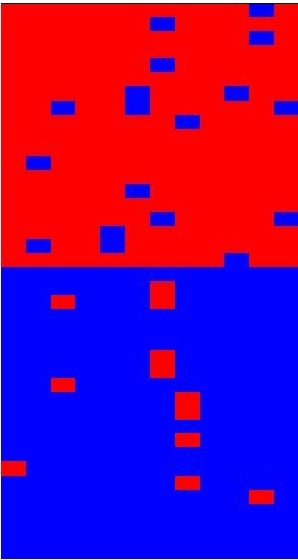} \quad &
\includegraphics[width=.2\textwidth]{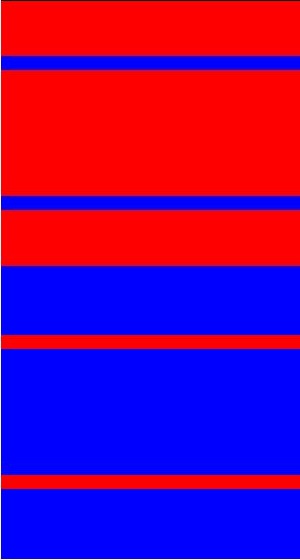} \\
$\mathcal{A}_{\e}$ & $\mathcal{G}_{\e}$ & $\Omega_{\e}$
\end{tabular}
\caption{Examples of the various good sets defined in Table~\ref{chart}. $\Gamma_{\e}$ is defined for technical purposes and is not included in the figure since it essentially looks like $\mathcal{G}_{\e}$ by Remark \ref{containment11}.
\label{goodsets}
}
\end{figure}

\begin{table}
\centering
\begin{tabular*}{0.732\textwidth}{ | c | c | c |}
  \hline  
  Set & Description & Defined in\\ 
  \hline 
  $\mathcal{A}_{\e}$ & all sites outside an $\e$-band  are the right color & \eqref{goodset11}\\
  \hline 
  $\mathcal{G}_{\e}$  &  at most $\e n^2$ sites are the wrong color  & \eqref{defdust}\\
  \hline
  $\Gamma_{\e}$  &  height function within $\e n^2$ of its maximum  & \eqref{goodset1}\\
  \hline
  $\Omega_{\e}$  & at most $\e n$ rows have sites of the wrong color  &  \eqref{optimalconf1}\\
  \hline
\end{tabular*}
\caption{Four different kinds of good sets.}
\label{chart}
\end{table}

\begin{definition}
For $\e>0$  let
\begin{equation}\label{goodset1}
\Gamma_{\e}:= 
\{\sigma \in \Omega: h(\sigma)> \bigl(\alpha(1-\frac{\alpha}{2})-\e\bigr)\!n^2\}.
\end{equation}
where $h(\cdot)$ is the height function defined in \eqref{wf}.
\end{definition}
\begin{remark}\label{containment11}
One easily observes the following inclusions:
\begin{enumerate}
\item $\Gamma_{\e^2} \subset  \mathcal{G}_{\e} \subset \Gamma_{\e}$
\item $\mathcal{A}_{\e} \subset  \Omega_{\e}$ 
\item $\mathcal{A}_{\e} \subset \mathcal{G}_{2\\e}.$
\end{enumerate}
\end{remark}

The following lemma states that starting from any configuration $\sigma$ 
the process takes $O(n^2)$ steps to hit $\Gamma_\e$. 
\begin{lemma}\label{closeopt1} Given any small enough $\e>0$ 
there exist positive constants $c=c(\e),d=d(\e),N=N(\e)$ such that 
for all $n>N$ and $\sigma\in \Omega$ 
\begin{eqnarray}\label{ht21}
\mathbb{P}_{\sigma}(\tau(\Gamma_{\e}) >2dn^2) & \le &e^{-cn^2}.
\end{eqnarray}
\end{lemma}
The next lemma asserts that once the process has hit $\Gamma_{\e}$ it is likely to stay in a neighborhood $\Gamma_{2\e}$ for a long time. 
\begin{lemma}\label{goodbad1} Given $\e>0$ there exist 
positive constants $c=c(\e),d=d(\e),N=N(\e)$ such that 
for all $n>N$ and $\sigma\in \Gamma_{\e}$ 
\begin{eqnarray}\label{ht31}
\mathbb{P}_{\sigma}(\tau(\Omega \setminus \Gamma_{2\e}) >e^{c n^2}) & \ge & 1- e^{-d n^2}.
\end{eqnarray}
\end{lemma}
The remainder of this section deduces Theorem~\ref{thmdust} from Lemmas~\ref{closeopt1} and~\ref{goodbad1}.
The Lemmas themselves are proved in Section \ref{phtr}.

\subsection{Bounding stationary measure in terms of hitting times}
We now prove a general lemma on Markov chains relating hitting times to the stationary distribution, which will be used in the proofs of both Theorems \ref{mainresult} and \ref{thmdust}. 
Roughly, it says that if one can identify subsets $A \subset B$ of the state space of a Markov chain such that (1) $A$ is hit quickly from all starting points, and (2) it takes a long time to escape $B$ from all starting points inside $A$, then the stationary distribution assigns high probability to $B$.
This along with Lemmas \ref{closeopt1} and \ref{goodbad1} would then allow us to conclude that $\Gamma_{2\e}$ has high stationary measure.
\begin{lemma}\label{hitstation} 
Let $\omega(\cdot)$ be an irreducible Markov chain on a finite state space $\mathcal{M}$. Suppose $A\subset B \subset \mathcal{M}$. Let $t_1,t_2,p_1,p_2$ be such that  
\begin{eqnarray}
\label{hyp1}
\max_{\omega\in \mathcal{M}}\bigl\{ \mathbb{P}_{\omega}(\tau(A) \ge  t_1)\bigr\} & \le & \ p_{1} \\
\label{hyp2}
\min_{\omega \in A} \bigl\{\mathbb{P}_{\omega}(\tau({B^c}) \ge t_2)\bigr\} & \ge & 1-\ p_2.
\end{eqnarray}
Then $$\pi(B^{c})\le \frac{t_1}{t_2}+p_1+p_2,$$
where $\pi$ is the stationary distribution of the Markov chain $\omega(\cdot)$ on $\mathcal{M}.$
 \end{lemma}
 
\begin{proof}
First we note that by the ergodic theorem for Markov chains 
for all $\omega \in \mathcal{M}$  
\begin{equation}\label{ergodbound}
\pi(B^\mathsf{c})=\lim_{t\rightarrow \infty}\frac{\E_{\omega}
\sum_{k=1}^{t}\mathbf{1}(\omega(k)\in B^\mathsf{c})}{t}.
\end{equation}
The proof of the above is standard. For example, it directly follows from  \cite[Theorem 4.9]{lpw}. 
We now claim that the following bound is true: 
$$\sup_{\omega\in \mathcal{M}}\E_{\omega}\sum_{k=1}^{t_2}\mathbf{1}(\omega(k)\in B^\mathsf{c})\le t_1+(p_1+p_2)(t_2).$$
To see this decompose the sum on the left hand side:
$$\E_{\omega}\sum_{k=1}^{t_2}\mathbf{1}(\omega(k)\in B^\mathsf{c})=\E_{\omega}\sum_{k=1}^{\tau(A)\wedge t_2}\mathbf{1}(\omega(k)\in B^\mathsf{c})+\E_{\omega}\sum_{k=\tau(A)\wedge t_2}^{t_2}\mathbf{1}(\omega(k)\in B^\mathsf{c}).$$
 The first sum is at most $t_1$ on the event that $\tau({A}) \le t_1$ and at most $t_2$ on the complement. The second sum is at most $t_2$ on the event that after $\tau({A})$ the process exits $B$ in fewer than $t_2$ steps and $0$ otherwise. Hence adding these up we get the upper bound 
\begin{equation}\label{ub12} 
 \sup_{\omega \in \mathcal{M}}\E_{\omega}\sum_{k=1}^{t_2}\mathbf{1}(\omega(k)\in B^\mathsf{c})\le t_1 +p_1 t_2 +p_2t_2.
 \end{equation}
We now decompose the sum $\sum_{k=1}^{t}\mathbf{1}(\omega(k)\in B^\mathsf{c})$ into blocks of length $t_2$, i.e.
$$\sum_{k=1}^{t}\mathbf{1}(\omega(k)\in B^\mathsf{c})=\sum_{k=1}^{t_2}\mathbf{1}(\omega(k)\in B^\mathsf{c})+\sum_{k=t_2+1}^{2t_2}\mathbf{1}(\omega(k)\in B^\mathsf{c}) + \ldots.$$
The theorem now follows from \eqref{ergodbound} by using \eqref{ub12} for each of the above sums on the  right hand side. \end{proof}

\subsection{Proof of Theorem \ref{thmdust}}
Notice that it suffices to prove that for given small enough $\e$ for all large enough $n$
\begin{equation}\label{suff1}
\pi(\Gamma_{2\e})\ge 1-e^{-cn^2}
\end{equation}
for some $c=c(\e)> 0.$ 
By the lower containment in Remark \ref{containment11} 
this proves that $$\pi(\mathcal{G}_{\sqrt{2\e}})\ge 1-e^{-cn^2}.$$
The proof now follows immediately from Lemma \ref{hitstation} 
with the following choices of the parameters:
\begin{eqnarray*}
A&= &\Gamma_{\e}\\
B&= &\Gamma_{2\e}\\
t_1&= &d_1n^2\\
t_2&= &e^{d_2n}\\
p_1 &= & e^{-c_1n^2}\\
p_2 &= &e^{-c_2n^2}
\end{eqnarray*}
where $c_1,c_2,d_1 ,d_2$ are chosen such that the hypotheses \eqref{hyp1} and \eqref{hyp2} of Lemma \ref{hitstation} are satisfied. Lemmas \ref{closeopt1} and \ref{goodbad1} allow us to do that. 
\qed

\section{The main technical result}\label{secmtr}
We now build towards the proofs of Lemmas \ref{closeopt1} and \ref{goodbad1}. The next result is one of the key technical ingredients of the paper. 
Recall (\ref{statespace}) and (\ref{statespace1}).
Given $\sigma \in \Omega \cup \Omega'$ 
for any  $a\in \{0,\frac{1}{n},\ldots, 1\}$ 
let us say that the line $y=a$ is a \emph{bad level} 
if there exists $\mathbf{x}=(x,a)$ such that 
$\sigma(\mathbf{x})$ has the wrong color. 
We now define the set of configurations $\sigma$ that have very few bad levels.
\begin{definition}\label{defoptimalconf}
Given  $\e>0$ for all $n$ let $\Omega_{\e}$ be 
the set of all configurations $\sigma\in \Omega$ such that 
\begin{eqnarray}\label{optimalconf1}
\# \{y: \sigma((j,y))=2 \text{ for some } (j,y)\in \Cyl_n,\,\,\mbox{ with } y\le \alpha \}
& \le & \e n \quad \text{ and} \\
\# \{y:  \sigma((j,y))=1 \text{ for some }  (j,y)\in \Cyl_n,\,\mbox{ with } y\ge \alpha \}
& \le & \e n.
\end{eqnarray}
Define $\Omega'_{\e}$ similarly.
\end{definition}

Recall the competitive erosion chain $\sigma_t$ 
defined in \eqref{e.thecoloring}. 
This section is devoted entirely to the proof of the following theorem:

\begin{theorem}\label{nonnegativedrift}
\moniker{Lyapunov Function}
Given $\e>0$ there exists $a=a(\e)>0$ 
such that 
for all $n$ sufficiently large and
for all $\sigma\in \Omega_{\e}^\mathsf{c}$
\begin{equation*}
\E(h(\sigma_1)-h(\sigma_0)\mid \sigma_0=\sigma)\ge a
\end{equation*} 
where $h(\cdot)$ is the height function defined in \eqref{wf}. 
\end{theorem}

For $A\subset \Cyl_n$ and $v \in \Cyl_n$ define
\begin{equation}\label{stoppedheightgeneral}
H_{A}(v)=\E_{v}[y_{\tau({A^c})}] 
\end{equation}
where $(x_{\tau({A^c})},y_{\tau({A^c})})$ is the point at which the random walk exits $A$ and $\E_{v}$ is the expectation under the random walk measure on $\Cyl_n$ with starting point $v.$

\begin{definition}\label{areadef2} 
Given $\sigma\in \Omega\cup \Omega'$ define $R_1 = R_1(\sigma)\subset \Cyl_n$ 
to be the set of all points of color $1$ in $\Cyl_n$ 
reachable by a monochromatic path of color $1$ in $\sigma$ from a point in  
$C_n \times \{0\}$. 

Similarly let $R_2 = R_2(\sigma)$ be the set of all points of color $2$ in $\Cyl_n$ 
reachable by a monochromatic path of color $2$ from a point in  
$C_n \times \{1\}$.
See Figure~\ref{figdiff}. 
\end{definition}
\begin{figure}
\centering
\includegraphics[scale=.6]{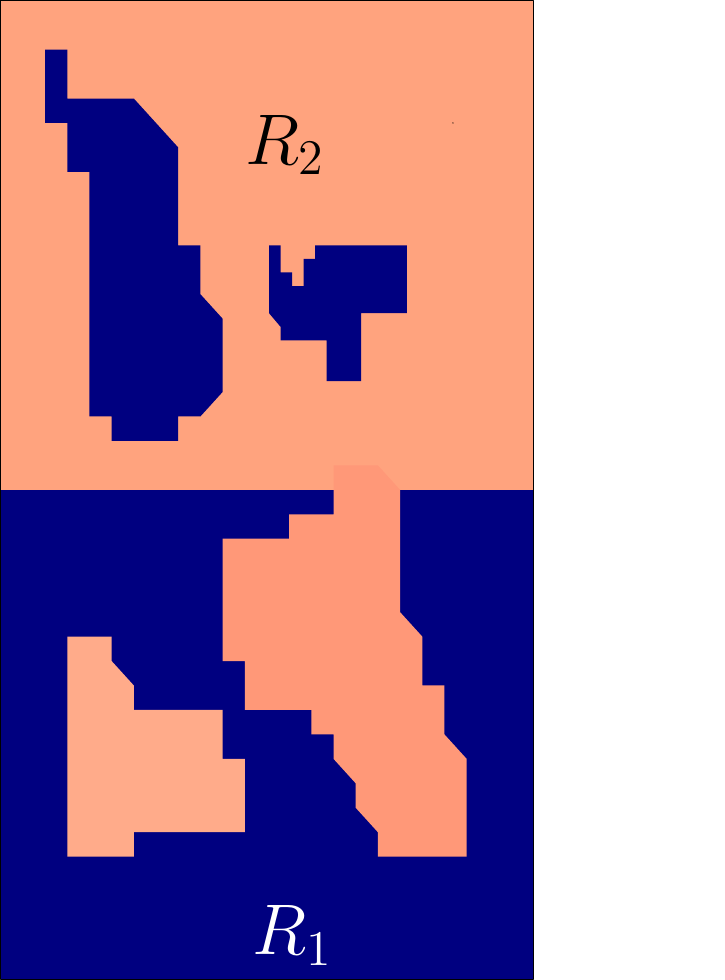}
\caption{The sets $R_1$ and $R_2$ of Definition \ref{areadef2}. 
Note that the two blue islands and the red island 
are not included in either set.}
\label{figdiff}
\end{figure}

\begin{lemma}\label{heightinter} For $\sigma \in \Omega,$
\begin{align}\label{heightexpress}
\E(h(\sigma_1)-h(\sigma_0)\mid \sigma_0 =\sigma) &= \frac{1}{n}\E_{\sigma_0} \left[\sum_{w\in C_n \times \{1\}}H_{R_{2}(\sigma_{1/2})}(w)\right]- \frac{1}{n}\sum_{v\in C_n \times \{0\}}H_{R_1(\sigma_0)}(v) \\
\end{align}
where $$\sigma_0,\sigma_{1/2},\sigma_{1}$$ 
form the steps of the competitive erosion chain. 
\end{lemma}

\begin{proof}
The first walker of color $1$ starts uniformly on $C_n \times \{0\}$ and stops on first hitting a site of color $2$, that is, on first exiting $R_1(\sigma)$. 
Writing \[ h(\sigma_1)-h(\sigma_0)=[h(\sigma_{1/2})-h(\sigma_{0})]+[h(\sigma_{1})-h(\sigma_{1/2})], \]
we see that the expected value of the first term in brackets on the right side is  
$$1- \frac{1}{n}\sum_{v\in C_n \times \{0\}}H_{R_1(\sigma_0)}(v).$$  
Now a walker of color $2$ starts uniformly on $C_n \times \{1\}$ 
and stops on first exiting $R_2(\sigma_{1/2})$, 
so the expected value of the second term in brackets is 
$$\displaystyle{\frac{1}{n}\E_{\sigma_0} 
\left[\sum_{w\in C_n \times \{1\}}H_{R_{2}(\sigma_{1/2})}(w)\right]-1}$$
(this term appears with an expectation because $\sigma_{1/2}$ is random). 
\end{proof}

\subsection{Energy and flows}\label{secef}
Next we relate the expression on the the right side of \eqref{heightexpress} to energy of flows.
We first recall some standard notation in electrical network theory.
On a graph $G=(V,E)$ let $w \sim v$ signify that $w$ is a neighbor of $v$.
Also let $\vec{E}$ be the set of directed edges where each edge in $E$ corresponds
to two directed edges in $\vec{E}$, one in each direction
(except for self-loops, 
which correspond to just one directed edge in $\vec{E}$).
A flow is an antisymmetric  function $f: \vec{E} \rightarrow \mathbb{R}$
(i.e.\ a function satisfying $f(w,v)=-f(v,w)$).
Note that by definition the value of a flow on a self-loop is $0.$
Define the energy of a flow $f$ by
\begin{equation}\label{enerno}
\mathcal{E}(f)=\frac12 \sum_{(v,w)\in \vec{E}}f(v,w)^2.
\end{equation} 
For any flow $f: \vec{E} \rightarrow \mathbb{R}$
 define the divergence $\div f: V \rightarrow \mathbb{R}$
by 
\begin{equation}\label{divdef}
\div f (v) = \sum_{w\sim v} f(w,v).
\end{equation}

Note that for any flow $$\sum_{x\in V}\div f=\sum_{\overset{x,y \in V}{y \sim x}}f(x,y)+f(y,x)=0$$ since $f$ is antisymmetric.
For disjoint subsets $A,B \subset V$ and a flow $f$ we say that the flow is from $A$ to $B$ if $\div f(x)=0$ for all vertices $x$ except vertices of $A$ and $B$ and the sum of the divergences across vertices of $A$ and $B$ are non-negative and non-positive respectively. For more about flows see \cite{lpw}*{Chapter 9}.
 
Now we state a key lemma relating Lemma \ref{heightinter} to energy of flows on $\Cyl_n.$
\begin{lemma}\label{heightgreen1} For any $A \subset \Cyl_{n}$  such that $A\cap (C_n \times \{1\})=\emptyset$, $$\frac{1}{n}\sum_{v\in C_n \times \{0\}}H_{A}(v)=\inf_{f}\mathcal{E}(f)$$
where the infimum is taken over all flows 
from $A \bigcap \left(C_n \times \{0\}\right)$ to $ A^c$
such that for $(x,y)\in A$
\begin{eqnarray*}
\div(f)(x,y)=\left\{\begin{array}{cl}
\frac{1}{n} & \mbox{if $y=0$,} \\
0 & \mbox{otherwise.} 
\end{array}
\right.
\end{eqnarray*}
Similarly 
for any $B \subset  \Cyl_{n}$  such that $B\cap (C_n \times \{0\})=\emptyset$,
$$1-\frac{1}{n}\sum_{v\in C_n \times \{1\}}H_{B}(v)=\inf_{f}\mathcal{E}(f)$$
where the infimum is taken over all flows 
from $(C_n \times \{1\}) \cap B$ to $ B^c$
such that for $(x,y)\in B$
\begin{eqnarray*}
\div(f)(x,y)=\left\{\begin{array}{cl}
\frac{1}{n} & \mbox{if $y=0$,} \\
0 & \mbox{otherwise.} 
\end{array}
\right. 
\end{eqnarray*}
\end{lemma}
The proof is deferred to Appendix \ref{app2}.
We now briefly sketch how it helps in proving Theorem \ref{nonnegativedrift}.

By  \eqref{heightexpress} and Lemma \ref{heightgreen1} one sees that 
\begin{equation*}
\E(h(\sigma_1)-h(\sigma_0)\mid \sigma_0=\sigma ) = 1-\mathcal{E}(f_1)-\E[\mathcal{E}(f_2)],
\end{equation*} where $f_1$ and $f_2$ are flows on $R_{1}(\sigma)$ and $R_{2}(\sigma_{1/2})$ satisfying the divergence conditions mentioned in Lemma \ref{heightgreen1} and have minimal energy.
Note that the second term on the RHS has an extra expectation to average over the random intermediate random configuration $\sigma_{1/2}.$ 

It turns out that the flow with minimal energy on $R_1$ is the ``natural" flow  induced by random walk started uniformly on $C_n \times \{0\}$ and killed on exiting $R_1$ and similarly for $R_2.$

Theorem \ref{nonnegativedrift} is now proved by estimating  $\mathcal{E}({f}_1)$ and $\mathcal{E}({f}_2).$ Consider the voltage function on $\Cyl_n$ that is linear in the height. $\mathcal{E}({f}_{1})$ can be interpreted as the voltage difference between $C_n \times \{0\}$ and the boundary of $R_1$ with glued boundary condition. The above heuristics are formalized in Appendix \ref{app2}. Now if the boundary is not horizontal enough, gluing vertices at different potentials causes a voltage drop. We estimate this voltage drop en route to proving Theorem \ref{nonnegativedrift} by constructing suitable flows on $R_1$ and $R_2.$

\subsection{Towards the proof of Theorem \ref{nonnegativedrift}}
For technical purposes we introduce a few definitions. Recall $R_1$ and $R_2$ from Definition \ref{areadef2}.
\begin{definition}\label{conn1}
Given $\sigma\in \Omega\cup \Omega'$ 
let $\tilde{R}_1$  be the set of all points in $\Cyl_n$ 
reachable by a monochromatic path of color $1$ in $\sigma$ from a point in  
$C_n \times \{0\}$ which does not hit the set $C_n \times \{1\}$. 

Similarly define $\tilde{R}_2.$ See Figure~\ref{bend1}.
\end{definition}

For $i=1,2$ define  
\begin{equation}\label{stopping1}
\tau_i=\tau(R_i^\mathsf{c})
\end{equation}
 to be the exit time from  $R_i$. Also let 
\begin{equation}\label{stopping2}
\tilde{\tau}_1=\tau(\tilde{R}_1^\mathsf{c})=\tau_{1}\wedge \tau (C_n \times \{1\}) 
\end{equation}
and similarly $\tilde{\tau}_2.$

\begin{remark}\label{monotoneheight}
Recalling the notation in \eqref{stoppedheightgeneral} we see that 
$$H_{{R}_{i}(\sigma)}(v)\le H_{\tilde{R}_{i}(\sigma)}(v)$$ since on the event $\tilde{\tau_1}<\tau_1,$ $y_{\tilde{\tau_1}}=1$ and hence $$y_{\tau_1}< y_{\tilde{\tau}_1}=1.$$ 
\end{remark}

For notational simplicity we write 
\begin{align*}
H_{1}(\cdot) &= H_{1,\sigma}(\cdot)= {H}_{R_{1}(\sigma)}(\cdot)\\
H_{2}(\cdot) &= H_{2,\sigma}(\cdot)= {H}_{R_{2}(\sigma)}(\cdot)\\
\tilde{H}_{1}(\cdot) &= \tilde{H}_{1,\sigma}(\cdot)= {H}_{\tilde{R}_{1}(\sigma)}(\cdot)\\
\tilde{H}_{2}(\cdot) &= \tilde{H}_{2,\sigma}(\cdot)= {H}_{\tilde{R}_{2}(\sigma)}(\cdot).
\end{align*}
We use the first notation when the underlying $\sigma$ is clear from context.

For every $k \in C_n,\sigma \in \Omega\cup \Omega'$ let
\begin{align}\label{heightboundary}
y_k^*(\sigma)&= \inf\{y:(k,y)\notin  \tilde{R}_1\}\\
\label{heightboundary1}
y_k^{**}(\sigma)&= \sup\{y:(k,y)\notin  \tilde{R}_2\}. 
\end{align}
See Figure~\ref{f.twobnd1}.
\begin{figure}{} 
\centering
\includegraphics[scale=.5]{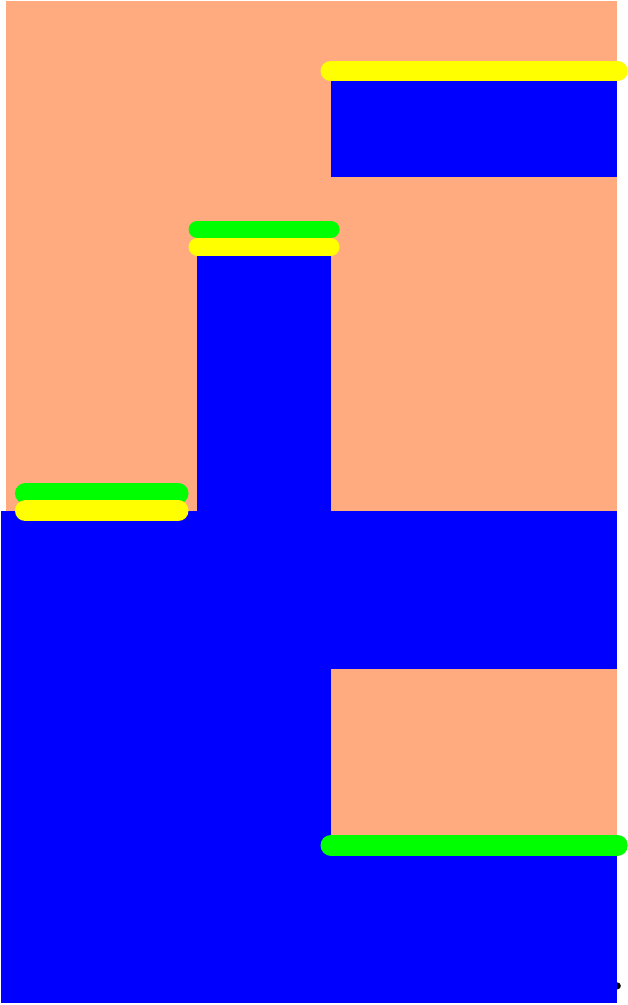}
\caption{ The green and yellow segments depict $y^{*}_k$ and $y^{**}_k$ respectively for $k \in C_n$.}
\label{f.twobnd1}
\end{figure}
\begin{figure}{} 
\centering
\includegraphics[scale=.6]{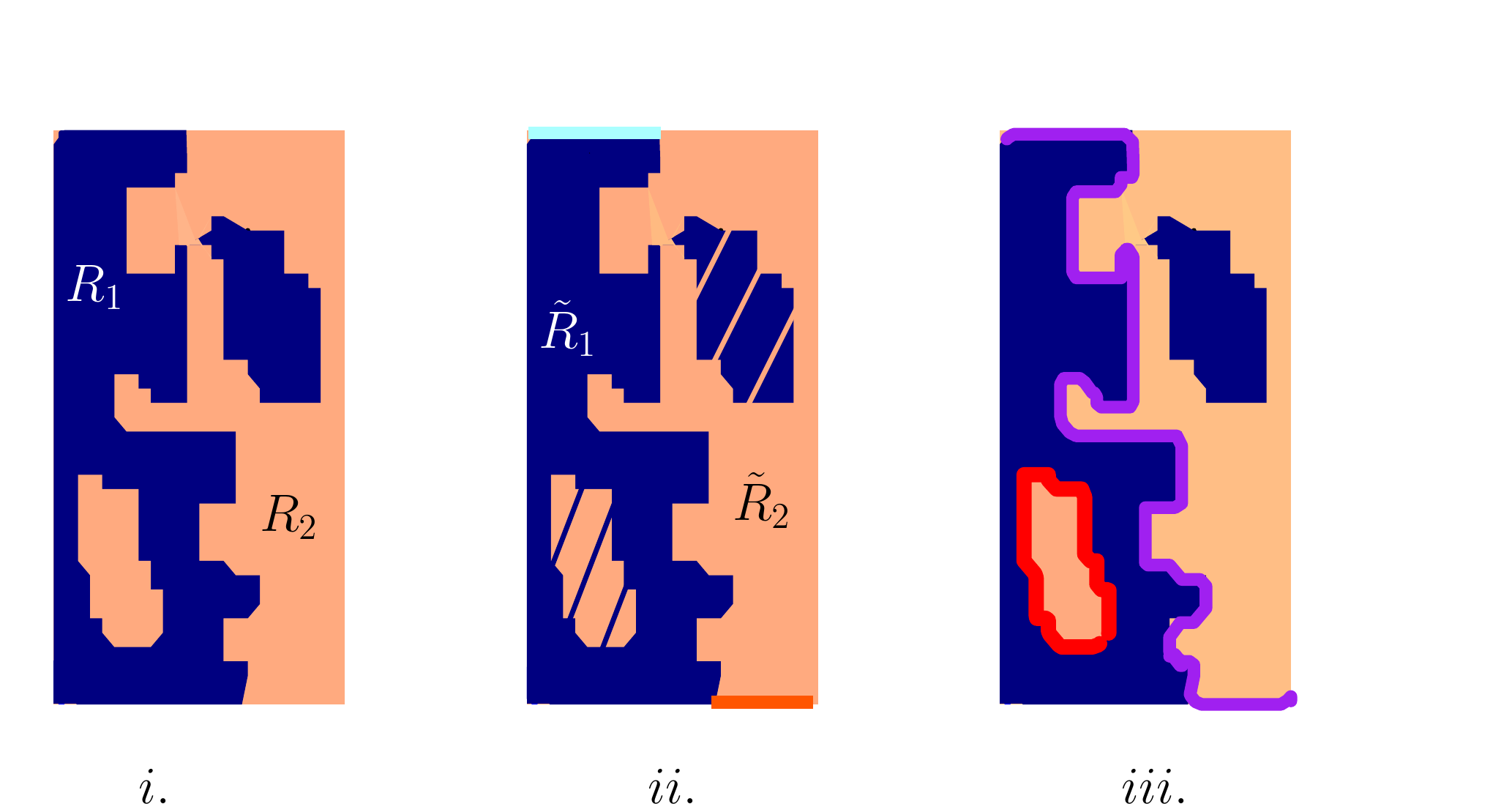}
\caption{$i.$ A general configuration with $R_1$ and $R_2$. Note than in $ii.$ the shaded red region and the top light-blue line ($R_1 \cap C_n\times \{1\}$) are outside $\tilde{R}_1$. Similarly the shaded blue region and  the bottom orange line are outside $\tilde{R}_2$. Notice the difference between $R_i$ and $\tilde{R}_i$. In $iii.$  $\partial^{out}{\tilde{R}_1}$ is represented by the purple curve. The rest of $\partial{\tilde{R}_1}$ (bottom left) is represented by the red curve. }
\label{bend1}
\end{figure}

As usual we will suppress the dependence on $\sigma$ and denote $y_k^{*}(\sigma)$, $y_k^{**}(\sigma)$ by $y_k^{*}$ and $y_k^{**}$ respectively when $\sigma$ is clear from the context.
\begin{lemma}\label{monotone} For $\sigma \in \Omega\cup \Omega'$
\begin{eqnarray}\label{mon1}
\frac{1}{n}\sum_{k \in C_n} H_{1}(k,0)&\le & \frac{1}{n} \sum_{k \in C_n}y_k^{*}\\
\label{mon2}
\frac{1}{n}\sum_{k \in C_n}H_{2}(k,0) & \ge & \frac{1}{n}\sum_{k \in C_n}y_k^{**}.
\end{eqnarray}
\end{lemma}
\begin{proof}We will only prove \eqref{mon1} since the arguments for \eqref{mon2} are symmetric.  By Remark \ref{monotoneheight} it suffices to prove that 
\begin{equation}\label{argument1}
\sum_{k \in C_n}\tilde{H}_{1}(k,0) 
\le \sum_{k \in C_n}y_k^{*}.
\end{equation}
Note that $\tilde{R}_1$ satisfies the hypothesis of Lemma \ref{heightgreen1}. Thus using the lemma we will prove the above by  finding a flow $f^*$  from $\tilde{R}_1\cap (C_n \times \{0\})$
 to $ \tilde{R}^c_1$ such that 
for all $(k,y)\in \tilde{R}_1$
\begin{eqnarray*}
\div(f^*)(k,y)=\left\{\begin{array}{cl}
\frac{1}{n} & \mbox{if $y=0$,} \\
0 & \text{ otherwise} 
\end{array}
\right.
\end{eqnarray*}
and  
\begin{equation}\label{argument4}
\mathcal{E}(f^*)= \frac{1}{n}\sum_{k \in C_n}y_k^{*}.
\end{equation}

Let
\begin{eqnarray}\label{flowdef1}
f^*((k,y),(k,y+\frac{1}{n}))&=&-\frac{1}{n}\mbox{ for } (k,y)\in \tilde{R}_1,\,\,\, 0 \le y < y_k^*
\end{eqnarray}
and let every other edge have zero flow.
It is easy to check that $f^*$ satisfies the required divergence conditions and \eqref{argument4}. See Figure~\ref{bend} $i.$
\end{proof}

The next lemma can be thought of as a weaker version of  Theorem \ref{nonnegativedrift}. It shows that at all times the height function has an ``almost" non-negative drift whereas  Theorem \ref{nonnegativedrift} states that outside $\Omega_{\e}$ the drift becomes significantly stronger.
\begin{lemma}\label{short1}For all $\sigma\in \Omega$
$$\E(h(\sigma_1)-h(\sigma_0)\mid \sigma_0=\sigma )\ge -\frac{1}{n}$$
where $h(\cdot)$ is the height function defined in \eqref{wf}.
\end{lemma}
Note that the above lemma 
is sharp for $\sigma_{0}=\sigma_{\max}$ \eqref{maxbnd}, 
up to a constant factor.

\begin{proof}
By Lemma \ref{heightinter}
$$
\E(h(\sigma_1)-h(\sigma_0)\mid \sigma_0=\sigma ) = 
\frac{1}{n}\E\left[\sum_{w\in C_n \times \{1\}}H_{2,\sigma_{1/2}}(w)\right]
-\frac{1}{n}\sum_{v\in C_n \times \{0\}}H_{1,\sigma_{0}}(v).
$$
Thus by Lemma \ref{monotone}   
\begin{equation}\label{recall1}
\E(h(\sigma_1)-h(\sigma_0)\mid \sigma_0=\sigma )\ge 
\frac{1}{n}\E\left[\sum_{k \in C_n}y_k^{**}(\sigma_{1/2})\right] 
-\frac{1}{n}\sum_{k \in C_n}y_k^{*}(\sigma_0)
\end{equation}
where $y_k^*(\cdot)$ and $y_k^{**}(\cdot)$ are defined in \eqref{heightboundary} and \eqref{heightboundary1} respectively. 
We make the following two easy observations:
For all $\sigma_0=\sigma\in \Omega$ and $k \in C_n,$
\begin{eqnarray}\label{2obs1}
y_k^{**}(\sigma_{1/2}) &\ge & y_k^{**}(\sigma_{0})\\
\label{2obs2}
y_k^{**}(\sigma_0)-y_k^*(\sigma_0) &\ge & -\frac{1}{n}.
\end{eqnarray}
The first one is a straightforward consequence of the fact that $\sigma_{1/2}$ has one more blue vertex than 
$\sigma_{0}$.
To see the second one, notice that for all  $k\in \frac{1}{n}\mathbb{Z}/{n\mathbb{Z}},$ either $y_k^*(\sigma_0)=0$ or $$\{(k,0),\ldots ,(k,y_k^*(\sigma_0)-\frac{1}{n})\}\in \tilde{R}_1.$$ 
In the first case \eqref{2obs2} follows since $y_k^{**}(\sigma_0)\ge 0$. 
In the second case,  $\tilde{R}_1$ and $\tilde{R}_2$ are disjoint and therefore $y_k^*(\sigma_0)-\frac{1}{n}\notin \tilde{R}_2$. Hence by definition \eqref{2obs2} follows.

Thus 
\begin{eqnarray}\label{arg7}
\E(h(\sigma_1)-h(\sigma_0)\mid \sigma_0=\sigma )&\ge 
& \frac{1}{n}\E\left[\sum_{k \in C_n}y_k^{**}(\sigma_{1/2})\right] 
-\frac{1}{n}\sum_{k \in C_n}y_k^{*}(\sigma_0)\\
& \ge & \frac{1}{n}\sum_{k \in C_n}y_k^{**}(\sigma_{0})- \frac{1}{n}\sum_{k \in C_n}y_k^{*}(\sigma_0)\\
\label{arg9}
&\ge & -\frac{1}{n}.
\end{eqnarray}
The first inequality is \eqref{recall1}, the second and third follow immediately from  the above two observations. Hence we are done.
\end{proof}

\subsection{Bending flows}\label{secpot1}
 \begin{figure}
\centering
\includegraphics[scale=.6]{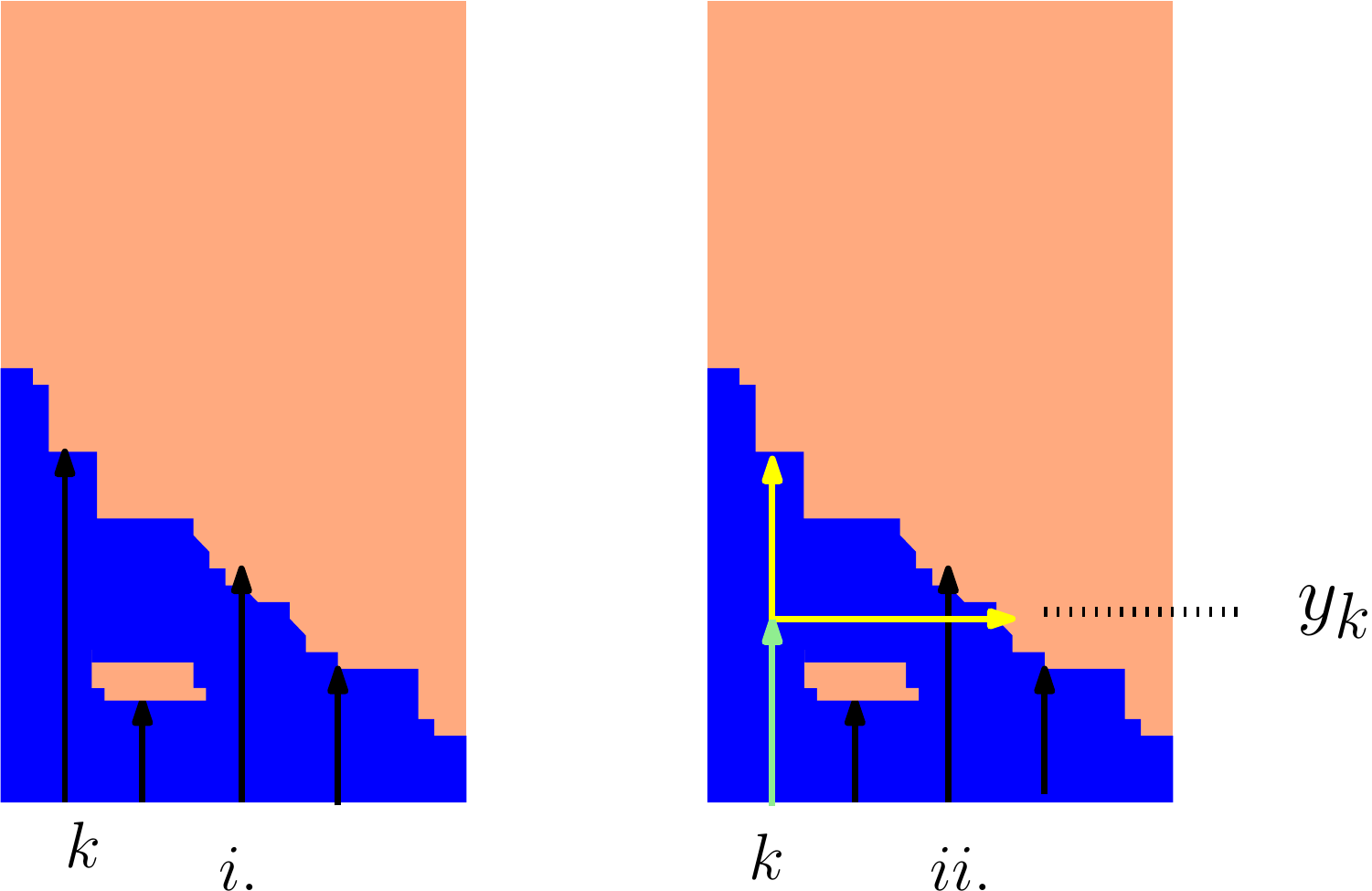}
\caption{Bending the trivial flow to improve upon Lemma \ref{short1}. In $i.$ every vertical edge carries mass $\frac{1}{n}.$ In $ii.$ the flow carries mass $\frac{1}{n}$ up to height $y_k$ where it splits into $\frac{d_k}{n}$ (vertically) and $1-\frac{d_k}{n}$ (horizontally). The $d_k$'s  are chosen  to be $\frac{1}{1+y_k^{*}-y_k}$. Also note the flows are split only for the columns where $y^*_{k}$ and $y^{**}_{k}$
are close to each other to beat \eqref{2obs2} used in Lemma \ref{short1}. \label{bend}}
\end{figure}

We prove Theorem \ref{nonnegativedrift} in this subsection by improving upon Lemma \ref{short1} which was 
 proved using the trivial flow on the cylinder graph that sends all the mass along the vertical edges. However if the boundary of $R_1$ (Definition \ref{areadef2}) is not horizontal then flows that send some mass along the horizontal edges are more optimal. This is because the horizontal boundary edges have significant harmonic measure and the random walk started from $C_n \times \{0\}$  has a significant chance of exiting $R_1$ through these edges. Thus to prove Theorem \ref{nonnegativedrift} we modify the trivial flow by sending some mass along the horizontal edges. See Figure~\ref{bend}. We begin with a few definitions.

For any $m\in (0,1)$ let us define the following sets:
\begin{eqnarray} \label{goods1}
W & =& \{\ell:y^{**}_{\ell}\ge m\}\\
\label{goods2}
W' & = &\{\ell:y^{*}_{\ell}\le m\}.
\end{eqnarray}
Also let us define the following properties for $c>0$ and $m\in (0,1)$:
\begin{prop}\label{fact1}
There exist $cn/4$ horizontal lines below the level $y=m-c/4$ 
that intersect $\Cyl_{n}\setminus {\tilde{R}_1}$.
\end{prop}
\begin{prop}\label{fact2}
There exist $cn/4$ horizontal lines above the level $y=m+c/4$ 
that intersect $\Cyl_n\setminus \tilde{R}_2$.
\end{prop}

The statement of the next lemma roughly says that if there are sufficiently many bad levels below or above a certain horizontal line then the height function experiences a strong drift.
The proof uses the idea of bending flows outlined above.

\begin{lemma}\label{maincondition}Given $c,d>0$ suppose $\sigma\in \Omega$ is such that one of the following holds:
for some $m,$
\begin{itemize} 
\item[$i.$] Property \ref{fact1} holds for $c$  and $|W|\ge dn$.
\item[$ii.$] Property \ref{fact2} holds for $c$ and $|W'|\ge dn$.
\end{itemize}
Then for large enough $n,$
\begin{equation}\label{shortcut}
\E(h(\sigma_1)-h(\sigma_0)\mid \sigma_0=\sigma)\ge a
\end{equation}
where $a=a(c,d)$.
\end{lemma}

\begin{proof}
Let $m$ be as in the statement of the lemma. 
Because of the obvious symmetry between the two conditions 
we will only discuss the proof of case $i.$ 
Without loss of generality we can assume $c=d$.

\begin{figure}
\centering
\includegraphics[scale=.6]{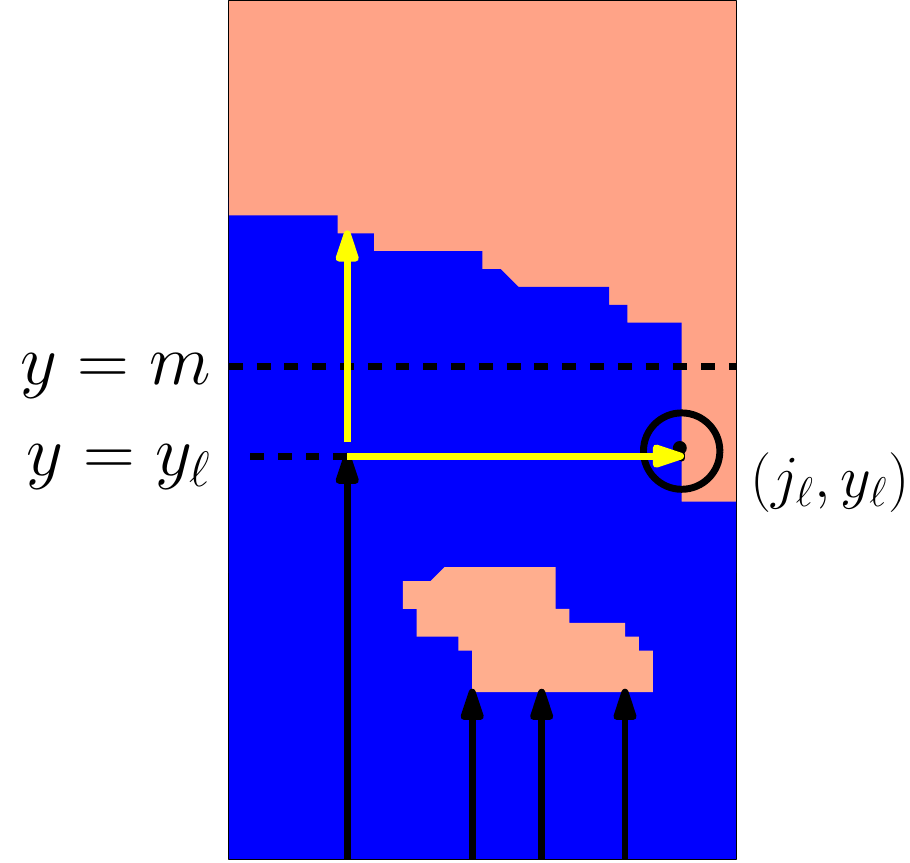}
\caption{Proof of Lemma \ref{maincondition} in case of Property $1$.
}
\label{prop}
\end{figure}

Let $W_1$ be an arbitrary subset of $W$ of size $cn/6$. 
We choose an arbitrary pairing of points in $W_1$ 
with horizontal lines mentioned in Property \ref{fact1}, 
i.e.\ for every $\ell$ in $W_1$ we associate 
a distinct horizontal line $y=y_{\ell}$ 
that intersects $\partial{\tilde{R}_1}$ and $y_{\ell}\le m-c/4$. 
Also let 
\begin{equation}\label{bdrypoint}
(j_{\ell},y_{\ell})
\end{equation}
 be the point on the line $y=y_{\ell}$ closest to the point $({\ell},y_{\ell})$ such that $(j_{\ell},y_{\ell})\in \partial\tilde{R}_1.$ $j_{\ell}$ exists since the line $y=y_{\ell}$ intersects $\partial{\tilde{R}_1}$ (see Figure~\ref{prop}).
Now for all $\ell \in W_1,$
\begin{equation}\label{lowergap}
y^{**}_{\ell}-y_{\ell} > c/4
\end{equation} 
as $W_1\subset W$.

Let 
\begin{equation}\label{settechnical}
W_2=\{\ell \in W_{1}:y^{*}_{\ell} \ge m-\frac{c}{8}\}.
\end{equation}
Hence for all $\ell \in W_1\setminus W_2$ 
\begin{equation}
\label{improvement1}
y^{**}_{\ell}-y^{*}_{\ell}\ge \frac{c}{8}.
\end{equation}

Thus we have 
$$\frac{1}{n}\sum_{k\in C_n}[y^{**}_k-y^*_k] \ge \frac{1}{n}\sum_{k\in W_1\setminus W_2}\frac{c}{8} -\frac{1}{n}.$$
The extra $-\frac{1}{n}$ term follows from \eqref{2obs2}.

Also for all $\ell \in W_2,$
\begin{equation}\label{improve0}
y^{*}_{\ell}-y_{\ell}\ge \frac{c}{8}.
\end{equation}

We now consider the two following sub-cases:\\

\noindent
(i) $|W_2|\le|W_1|/2.$\\

\noindent
 Then since $|W_1|= \frac{cn}{6}$ and hence $|W_1\setminus W_2|\ge \frac{cn}{12}$, by \eqref{improvement1} we have
\begin{equation}\label{easybound1}
\frac{1}{n}\sum_{k\in C_n}[y^{**}_k-y^*_k]\ge \frac{c^2}{96}-\frac{1}{n}.
\end{equation}
\noindent 
Thus in this case by \eqref{arg7},
$$\E(h(\sigma_1)-h(\sigma_0)\mid \sigma_0=\sigma )\ge \frac{c^2}{96}-\frac{1}{n} 
$$
and hence we are done.
\\

\noindent
(ii)\, $|W_2|\ge|W_1|/2$. \\

\noindent
The proof in this case is more involved than the previous case. 
We first claim that in this case 
there exists a flow $f_{1}$ from $C_n \times \{0\}\cap \tilde{R}_1$
to $\Cyl_n \setminus \tilde{R}_1$ such that 
for all $(k,y)\in \tilde{R}_1$
\begin{eqnarray}\label{divcond1}
\div(f_{1})(k,y)=\left\{\begin{array}{cl}
\frac{1}{n} & \mbox{if $y=0$,} \\
0 & \text{ otherwise} 
\end{array}
\right.
\end{eqnarray}
and
\begin{equation}\label{argument6}
\mathcal{E}(f_{1})\le \frac{1}{n}\sum_{k \in C_n}y_k^{*}-\frac{c^{3}}{1536}.
\end{equation} 

Before proving the above claim we show why it suffices and implies \eqref{shortcut}.
\noindent
By taking the set $A$ in Lemma \ref{heightgreen1} to be $\tilde{R}_1$ followed by using Remark \ref{monotoneheight} we get  
\begin{eqnarray}\label{implbound}
\frac{1}{n}\sum_{v\in C_n \times \{0\}}H_{1,\sigma_0}(v) & \le & \frac{1}{n}\sum_{v\in C_n \times \{0\}}\tilde H_{1,\sigma_0}(v)\\
\nonumber 
&\le & \frac{1}{n}\sum_{k \in C_n}y_k^{*}-\frac{c^{3}}{1536}.
\end{eqnarray}
By Lemma \ref{heightinter}
\begin{equation}\label{concl1}
\E(h(\sigma_1)-h(\sigma_0)\mid \sigma_0=\sigma ) =\frac{1}{n}\E[\sum_{w\in C_n \times \{1\}}H_{2,\sigma_{1/2}}(w)]- \frac{1}{n}\sum_{v\in C_n \times \{0\}}H_{1,\sigma_0}(v).
\end{equation}
The first term is bounded from below by $\displaystyle{\frac{1}{n}\sum_{k \in C_n}y_{k}^{**}(\sigma_0)}$ (as in the proof of \eqref{arg7}) and the second term
by \eqref{implbound} is upper bounded by $$\frac{1}{n}\sum_{k \in C_n}y^*_{k}(\sigma_0)-\frac{c^3}{1536}.$$

Hence we get
\begin{align*}
\E(h(\sigma_1)-h(\sigma_0)\mid \sigma_0=\sigma )& \ge \left[\frac{1}{n}\sum_{k \in C_n}y_k^{**}(\sigma_{0})-\frac{1}{n}\sum_{k \in C_n}y^*_{k}(\sigma_0)\right]+ \frac{c^3}{1536}\\
&\ge \frac{c^3}{1536}-\frac{1}{n}. 
\end{align*}
Note that the term in the brackets is greater than $-\frac{1}{n}$ by \eqref{2obs2}. Thus the proof of Lemma \ref{maincondition} is complete in this case except for the proof of the initial claim. 
We now  prove the claim by defining the flow $f_1$ as shown in Figure~\ref{prop}. Recall \eqref{settechnical}. For $k \notin W_2$ let
$$
f_{1}((k,y),(k,y+\frac{1}{n}))= -\frac{1}{n}\ \mbox{for all } (k,y)\in \tilde{R}_1 \mbox{ and } y< y_k^*. 
$$
For $k \in  W_2$ 
\begin{eqnarray*}
f_{1}((k,y),(k,y+\frac{1}{n}))&= &\left\{\begin{array}{ccc}
-\frac{1}{n} & \text{ if } & y < y_k  \\\\
-\frac{d_k}{n} & \text{ if } & y_k \le y < y_k^*
\end{array}
\right.\\
f_{1}((j,y_k),(j+\frac{1}{n},y_k))&= &- \frac{1-d_k}{n}\,\, \forall j=k,k+\frac{1}{n},\dots j_k 
\end{eqnarray*}
where $j_k$ was defined in \eqref{bdrypoint} and $0<d_k<1$ are to be specified later.
$f_1$ is $0$ on all other edges.
It is easy to see that the flow $f_{1}$ satisfies the divergence conditions \eqref{divcond1}.

To see why \eqref{argument6} is true first observe that
$$
\mathcal{E}(f_1)\le  \frac{1}{n}\sum_{k \in W_2 } [y_k+ d_k^2 (y_k^{*}-y_k)+(1-d_k)^2]+\frac{1}{n}\sum_{k \notin W_2 }y^*_{k}.
$$
This is because by construction: 
\begin{itemize}
\item for $k\notin W_2$ the flow along the line $x=k$ is $\frac{1}{n}$ for $ny_k^{*}$ edges.
\item   for $k \in W_2$
\begin{itemize}
\item
the flow along the line $x=k$ is $\frac{1}{n}$ for $ny_k$ edges and $\frac{d_k}{n}$ for $n(y_k^{*}-y_k)$ edges. 
\item  the number of edges with flow $\frac{1-d_k}{n}$ on the line $y=y_k$ is at most $n$.
\end{itemize}
\end{itemize}
Plugging in $d_k=\frac{1}{1+y_k^{*}-y_k}$  we get that the expression on the RHS equals
\begin{eqnarray*}
& &
\frac{1}{n}\sum_{k \in W_2 } \left[y_k+ \frac{ (y_k^{*}-y_k)}{(y_k^{*}-y_k)+1}\right]+\frac{1}{n}\sum_{k \notin W_2 }y^*_{k} \\
&=&\frac{1}{n}\sum_{k \in W_2 } \left[y^*_k + (y_k-y^*_k)+ \frac{ (y_k^{*}-y_k)}{(y_k^{*}-y_k)+1}\right]+\frac{1}{n}\sum_{k \notin W_2 }y^*_{k},\\
&=& \frac{1}{n}\sum_{k \in W_2 } y^*_k + +\frac{1}{n}\sum_{k \notin W_2 }y^*_{k} + \frac{1}{n}\sum_{k \in W_2 } \left[(y_k-y^*_k)+ \frac{ (y_k^{*}-y_k)}{(y_k^{*}-y_k)+1}\right] \\
&=&\frac{1}{n}\sum_{k \in C_n}y^*_{k}-\frac{1}{n}\sum_{k \in W_2 }  \left[\frac{ (y_k^{*}-y_k)^2}{(y_k^{*}-y_k)+1}\right].
\end{eqnarray*}

Now since  $|W_1|\ge \frac{cn}{6}$ and  we are considering the case $|W_2|\ge \frac{|W_1|}{2}$,  we have $|W_2|\ge \frac{cn}{12}$. Also for $k\in W_2$ by \eqref{improve0} $$\frac{c}{8}\le y_k^{*}-y_k\le 1.$$ 
Plugging these in the expression $$\frac{1}{n}\sum_{k \in W_2 }  [\frac{ (y_k^{*}-y_k)^2}{(y_k^{*}-y_k)+1}]$$
gives us \eqref{argument6}. Thus the proof of Lemma \ref{maincondition} for case $i.$ is complete. 
\end{proof}

The proof of Theorem \ref{nonnegativedrift} involves considering a few cases and showing that the hypotheses of Lemma \ref{maincondition} are satisfied in each case. We start with some definitions.

Let $\Cyl_n^+$ be the slightly larger graph $C_n \times P_{n+2}$ where we identify the path $P_{n+2}$ with  
	\[ \{ \frac{a}{n} \,:\, a=-1,0,1,\ldots,n+1 \}. \]
Recall the definition \eqref{e.boundary} of the boundary of a set of vertices. The following is a technical definition needed for convenience
(see Figure~\ref{bend1}):
\begin{definition}\label{outbound}Let $\partial \tilde{R}_1$ be the boundary of 
$\tilde{R}_1\cup \bigl(C_n \times \{\frac{-1}{n}\}\bigr )$ in $\Cyl_n^+$.  
Define $\partial^{out} \tilde{R}_1$ to be the subset of $\partial \tilde{R}_1$
that is visible from $C_{n}\times \{1\}$, i.e., 
the set of all points in $\partial \tilde{R}_1$ 
that are connected to $C_n \times \{1\}$ 
in the subgraph induced by  $\Cyl_n\setminus \tilde{R}_1$. 
As usual we have similar definitions 
for $\partial \tilde{R}_2$, $\partial^{out} \tilde{R}_2$.  
\end{definition}

Note that since $\tilde{R}_1\cap \bigl (C_{n} \times \{1\}\bigr )=\emptyset,$ $$\partial \tilde{R}_1\subset \Cyl_n.$$

Define the graph $\Cyl_n^{*}$ to be the graph $\Cyl_n$ 
along with the additional diagonal edges 
\begin{eqnarray*}
\left\{\left(\frac{i}{n},\frac{j}{n}\right),
\left(\frac{i+1}{n},\frac{j+1}{n}\right)\right\}\\
\left\{\left(\frac{i}{n},\frac{j}{n}\right),
\left(\frac{i-1}{n},\frac{j+1}{n}\right)\right\}
\end{eqnarray*}
for all $i=0,\ldots n-1,\,j=0\ldots n-1$ where the addition in the first coordinate is in $(\frac{1}{n}\Z)/\Z.$
Call a subset $B\subset \Cyl_n,$ $*$-connected if it is a connected set in the graph $\Cyl_n^{*}$.
\begin{remark}\label{connproof}
It follows from \cite[Lemma 2]{at1} that for $i=1,2$, 
$\partial^{out}{\tilde{R}_i}$ is a $*$-connected set. 
\end{remark}
Let 
\begin{eqnarray}\label{heightset1}
\mathcal{Y}&=&\{ y: (k,y) \in \partial^{out}{\tilde{R}_1} \text{ for some } k\in C_n\}.
\end{eqnarray}
 Note that $\mathcal{Y}$ is a connected subset of $[0,\frac{1}{n}\dots,1 ]$ since $\partial^{out}{\tilde{R}_1}$ is $*$-connected.
Also for every $k \in C_n$
let 
\begin{eqnarray}\label{bounddef}
y_{(k)}&=&\inf\{y:(k,y) \in \partial^{out}{\tilde{R}_1} \}\\
\nonumber
y^{(k)}&=&\sup\{y:(k,y) \in \partial^{out}{\tilde{R}_1} \}.
\end{eqnarray}
\begin{figure}
\centering
\includegraphics[scale=.5]{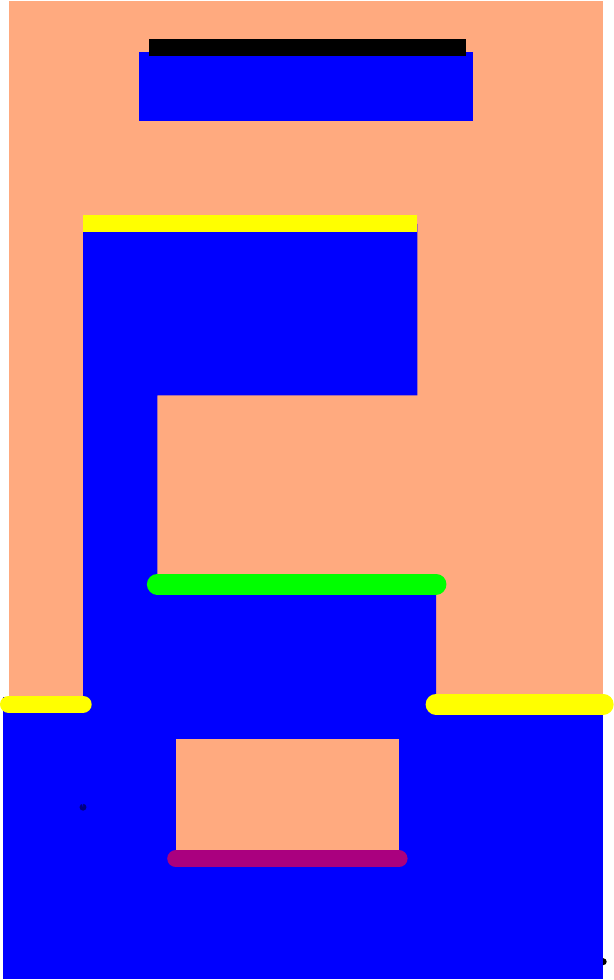}
\caption{ The yellow segment depicts $y^{(k)}$ for $k \in C_n$. The green segment depicts $y_{(k)}$ for those $k$ for which $y^{(k)}\neq y_{(k)}$. Note that $y^{(k)}=y_{(k)}$ unless $\partial^{out}{\tilde{R}_1}$ intersects the column $x=k$ at more than one point.  Also we include the purple segment at the bottom and the black segment at the top depicting $y_k^{*}$ and $y_K^{**}$ respectively to illustrate the differences between the four definitions.
}
\label{f.twobnd}
\end{figure}
See Figure~\ref{f.twobnd}.
Recall the definition of $y_k^*$ and $y_k^{**}$ from \eqref{heightboundary},\eqref{heightboundary1}. 
Note that for every $k \in C_n,$ 
\begin{equation}\label{triv1}
y_k^* \le  y_{(k)}.
\end{equation}
We make another simple observation:
For any $k\in C_n$ there exists $\ell \in \{k-\frac{1}{n},k,k+\frac{1}{n}\}$ such that 
\begin{equation}\label{triv2}
y^{(k)}- \frac{1}{n} \le  y_{\ell}^{**}.
\end{equation}
This follows since the point $(k,y^{(k)})\in \partial^{out} \tilde{R}_1$ 
and hence either  $y^{(k)}=0$, 
or $(k,y^{(k)})$ has a neighbor colored $1$ in $\Cyl_n$.

We are now ready to prove Theorem \ref{nonnegativedrift}. 
As we remarked earlier it suffices to show that 
the hypotheses of Lemma \ref{maincondition} are satisfied.\\

\noindent
{\bf{Proof of Theorem \ref{nonnegativedrift}.}}
By hypothesis $\sigma\in \Omega\setminus \Omega_{\e}.$ Also observe that for
any $\sigma \in \Omega$ one of the following three conditions holds:
\begin{enumerate}
\item[(i)] 
$\sup \mathcal{Y} -\inf \mathcal{Y}>\e/20.$
\item[(ii)]
$\sup \mathcal{Y} -\inf \mathcal{Y}\le\e/20$ and $|\sup \mathcal{Y}-\alpha|>\e/10$
\item[(iii)]
$\sup \mathcal{Y} -\inf \mathcal{Y}\le \e/20$ and $|\sup \mathcal{Y}-\alpha|\le \e/10$
\end{enumerate}
where the set $\mathcal{Y}$ was defined in \eqref{heightset1}.

\noindent
Case (i): Let $m = \frac{1}{2} (\sup{\mathcal{Y}}+\inf{\mathcal{Y}}),c=\frac{\e}{20}$.
By hypothesis
$\sup \mathcal{Y} -\inf \mathcal{Y} >c$. 
We first show under this hypothesis 
both Properties \ref{fact1} and \ref{fact2} hold.

By Remark \ref{connproof}, $\mathcal{Y}$ is connected.
Thus there exist $cn/4$ values in $\mathcal{Y}$ less than $m-c/4$. 
In other words:
\begin{center}
Property \ref{fact1} holds with  $m$ and $c$ as defined above.
\end{center}

\noindent
Again using connectedness of $\mathcal{Y}$,  $cn/4$ values in $\mathcal{Y}$ are greater than $m+c/4.$ Also notice that for every $k< \sup \mathcal{Y}$ the line $y=k$ must intersect $R_1$ since the line $y=\sup \mathcal{Y}$ has a vertex lying on the boundary of $R_1$ and hence we have a monochromatic path of color $1$ from the level $y=0$ to $y=\sup \mathcal{Y}- \frac{1}{n}.$  Thus there exist $cn/4$ horizontal lines above the level $y=m+c/4$ 
that intersect $\Cyl_n\setminus \tilde{R}_2$. Hence
\begin{center}
Property \ref{fact2} holds with  $m$ and $c$ as defined above.
\end{center}
\noindent
To verify the rest of the hypotheses of Lemma \ref{maincondition} 
we have to show that either the set $W$ or $W'$ is large.
Recall $y^{(k)}$ from \eqref{bounddef}. 
Clearly at least one of the inequalities
\begin{equation}
\label{condition1}
|\{ k:y^{(k)} >  m \}|  \ge  n/2 ,
\end{equation}
\begin{equation}
\label{condition2}
|\{k: y^{(k)} \le m \}| \ge  n/2
\end{equation}
\noindent
is true.
Recall from \eqref{goods1}  and \eqref{goods2}
\begin{eqnarray*}
W&=&\{\ell:y^{**}_{\ell}\ge m\}\\
W'&=&\{\ell:y^{*}_{\ell}\le m\}
\end{eqnarray*}
\noindent
When \eqref{condition1} holds,
by \eqref{triv2}  $$|W|\ge \frac{n}{6}.$$
\noindent
Similarly when \eqref{condition2} holds by \eqref{triv1}, $$|W'|\ge \frac{n}{2}.$$
\noindent
Thus all the hypotheses of Lemma \ref{maincondition} are satisfied in case (i) and hence we are done.

To verify the hypotheses of Lemma \ref{maincondition} in the remaining cases 
we start by making an observation. 
Let 
\begin{eqnarray*}
W''&:=&\{\ell:y^{**}_{\ell}\ge \inf \mathcal{Y}-\frac{1}{n}\}\\
W'''&:= &\{\ell \in C_n:  y^{*}_{\ell}  \le  \sup \mathcal{Y}\}.
\end{eqnarray*}
Now we have
\begin{eqnarray}\label{arg10}
|W''|&\ge & \frac{n}{3}, \\
\label{arg11}
|W'''|& = & n
\end{eqnarray}
(the first follows from \eqref{triv2} and the second from \eqref{triv1}).
Thus by choosing $m=\inf \mathcal{Y}$ or $\sup \mathcal{Y}$ the conditions for sets $W$ and $W'$ respectively in the hypotheses of Lemma \ref{maincondition} are always satisfied . 

The remainder of the proof shows that in cases (ii) and (iii) either Property \ref{fact1} or \ref{fact2} holds by choosing $m=\inf \mathcal{Y}$ or $\sup \mathcal{Y}$. This would then complete the proof by Lemma \ref{maincondition}.

Both cases (ii) and (iii) have two subcases, arguments for which are symmetric. We will discuss only 
one of the cases. In case (ii)  we will discuss the case $$\sup \mathcal{Y}-\alpha \ge \e/10$$ and skip the case 
$\sup \mathcal{Y}-\alpha \le -\e/10$.
In the case $\sup \mathcal{Y}-\alpha \ge \e/10$ $$\inf \mathcal{Y}-\alpha \ge \frac{\e}{10}-\frac{\e}{20}=\frac{\e}{20}.$$ 
Now this implies that there are at least $\frac{\e}{40}n^2$ red vertices below the line $y=\inf \mathcal{Y}- \frac{\e}{40}$  since $$\inf \mathcal{Y}- \frac{\e}{40}\ge \alpha + \frac{\e}{40}.$$
In particular there exist at least $\frac{\e}{40}$ disjoint horizontal lines below the line $y=\inf \mathcal{Y}- \frac{\e}{40}$ with at least one red vertex, i.e.\ Property \ref{fact1} is true for $m=\inf \mathcal{Y}$ and $c=\frac{\e}{40}.$

In case (iii) we use the fact that $\sigma \in \Omega \setminus \Omega_{\e}$ and hence there are at least $\e n$ horizontal lines with a red vertex below the line $y=\alpha$ or $\e n$ horizontal lines with a blue vertex above the line $y=\alpha$. We will discuss only the first case. 
Now by hypothesis in (iii)$$\inf \mathcal{Y} \ge \alpha -\frac{\e}{10}-\frac{\e}{20} \ge \alpha -\frac{\e}{5}.$$ 
Hence there exist at least $\frac{\e}{20}$ disjoint horizontal lines below the line $y=\inf \mathcal{Y}- \frac{\e}{20}$ with at least one red vertex, i.e.\ Property \ref{fact1} is true for $m=\inf \mathcal{Y}$ and $c=\frac{\e}{40}.$

The symmetric arguments which we skip will show that Property \ref{fact2} holds in the other cases with $m=\sup \mathcal{Y}$ and $c=\frac{\e}{40}$.
Thus we are done. 
 \qed

\section{Proofs of hitting time results}\label{phtr}
In this section we provide the proofs 
of Lemmas \ref{closeopt1} and \ref{goodbad1}.
We first state the following general lemma 
about hitting times of submartingales. 
The statement involves a few parameters 
and can be slightly difficult to follow.
However it will be useful in subsequent applications. 
Let $\omega(t)$ be a stochastic process 
taking values in an abstract set $\mathcal{D}.$ 
Also let $g:\mathcal{D} \rightarrow \R$ be a real-valued function. 
Let $\mathcal{F}_t$ be the filtration generated 
by the process $\omega(t)$ up to time $t$. 
Also define $X_t:=g(\omega(t)).$
\begin{lemma}\label{lemm:azuma1} 
Let $A_1,A_2,a_1 >0$. Suppose   
\begin{eqnarray}\label{hyp111}
|g(\omega)| & \le & A_1 \mbox{ for all } \omega \in \mathcal{D} \\
\label{hyp222}
|X_{t}-X_{t-1}| & \le & A_2 \mbox{ for all } t.
\end{eqnarray}
Also suppose that $B \subset \mathcal{D}$ is such that for any time $t$ 
\begin{eqnarray}\label{submart1}
\E(X_t-X_{t-1}|\mathcal{F}_{t-1})& \ge & a_1 \mathbf{1}\left(\omega_{t-1}\notin B\right).
\end{eqnarray} 
Then:
\begin{itemize}
\item [i.]  For any $a_2>0,$ 
$$P_{\omega}(\tau(B) \ge T) 
\le \exp\left({-\frac{{(a_2-a_1T)}^2}{4A_2^2T}}\right)$$  
for all $\omega \in \mathcal{D}$ such that $ g(w)\ge A_1-a_2$ 
and any $T$ such that $a_2-a_1T <0$. 
\item [ii.]  Now consider the special case when $B$ is a level set, 
i.e.\ suppose for some $a_4>2A_2$, $B=\{\omega: g(w)\ge A_1-a_4\}$ 
and $B'=\{\omega: g(w)\le A_1-2a_4\}$ .  
Then for all $\omega \in B$ and all $T> \frac{2A_2}{a_1}$
$$\mathbb{P}_{\omega}(\tau(B') \ge T') \ge  
1-\left[\exp\left({-\frac{a_4^2 }{32A_2^2T}}\right)
+\exp\left({-\frac{a_1^2 T^2}{32A_2^2T}}\right)\right],$$  
where $T' =\exp\left({\frac{\min( a_4^2,a_1^2 T^2 )}{32A_2^2T}}\right)$.
\end{itemize} 
 \end{lemma}
The proof of the above lemma follows easily from 
the standard Azuma-Hoeffding inequality for submartingales. 
We defer it to Appendix \ref{newapp}. 
However we immediately see some applications.
Recall the definition of $\Omega_{\e}$ from Table \ref{chart}.

\begin{lemma}\label{linehit} Given any $\e>0$ there exist positive constants
$c=c(\e),d=d(\e),N=N(\e)$ such that for all $n>N$ and $\sigma\in \Omega$ 
\begin{eqnarray}\label{ht1}
\mathbb{P}_{\sigma}(\tau(\Omega_{\e}) >dn^2) & \le &e^{-cn^2}.
\end{eqnarray}
\end{lemma}
\begin{proof}
The proof follows from Lemma \ref{lemm:azuma1} $i.$
The stochastic process we consider is the competitive erosion chain $\sigma_{t}$
and $X_{t}=h(\sigma_{t})$ is the height function defined in \eqref{wf}. 
We make the following choice of parameters: 
\begin{align*}
B &= \Omega_{\e}\\
A_1 &= n^2\\
A_2 &= 1\\
a_1 &= a(\e)\mbox{ appearing in Theorem }\ref{nonnegativedrift} \\
a_2 &= n^2\\
T &= \frac{2n^2}{a_1}.
\end{align*}
Clearly \eqref{hyp1} and \eqref{hyp2} are satisfied by \eqref{maxbnd} and \eqref{gradbound}.
Thus by Lemma \ref{lemm:azuma1} $i.$

$$\mathbb{P}_{\sigma}\left(\tau(\Omega_{\e})\ge \frac{2n^2}{a_1}\right) \le e^{-\frac{a_1n^2}{8}}.$$

\end{proof}

\subsection{Proof of Lemma \ref{closeopt1}}\label{seccloseopt1}
The proof follows from Lemma \ref{linehit} and the containment (Remark \ref{containment11})
$\Omega_{\e}\subset \Gamma_{\e}. 
$
\qed\\
%

\subsection{Proof of Lemma \ref{goodbad1}}\label{secgoodbad1}
The proof follows from Lemma \ref{lemm:azuma1} $ii$. Let $$X(t)=h(\sigma_t)$$ where $h(\cdot)$ is the height function defined in \eqref{wf}. 
We make the following choice of parameters:
\begin{align*}
B &= \Gamma_{\e}\\
B' &=  \Gamma_{2\e}\\
A_1 &= \alpha(1-\alpha/2)n^2\\
a_4 &= \e n^2 \\
A_2 &= 1\\
a_1 &= a(\e)\mbox{ appearing in Theorem }\ref{nonnegativedrift} \\
T &= n^2.
\end{align*}
The second containment in Remark \ref{containment11} along with Theorem \ref{nonnegativedrift} satisfy the drift condition \eqref{submart1} with $a_1=a(\e)$. 
Now by the above choice of parameters $T'=\exp\left({\frac{\min( \e^2,a_1^2 )n^2}{32}}\right).$
Thus by Lemma \ref{lemm:azuma1} $ii.$ for all $\sigma \in \Gamma_{\e},$
$$\mathbb{P}_{\sigma}(\tau(\Gamma_{2\e})\ge T')\ge  1-[\exp\left({-\frac{\e^2 n^2 }{32}}\right)+\exp\left({-\frac{a_1^2 n^2}{32}}\right)].$$
Thus the proof is complete.
\qed

\section{Proof of the main result}\label{2}
The proof of Theorem \ref{mainresult} has the same structure as the proof of Theorem \ref{thmdust}.
We first state and prove the following two results analogous to Lemmas \ref{closeopt1} and \ref{goodbad1}. 
\begin{lemma}\label{linehit2} Given any $\e>0$ there exist
positive constants $c=c(\e),d=d(\e),$ such that 
for all large enough $n$ and $\sigma\in \Omega_{\e}$ 
\begin{eqnarray}\label{ht11}
\mathbb{P}_{\sigma}(\tau(\mathcal{A}_{\sqrt{\e}}) >dn^2) & \le &e^{-cn}.
\end{eqnarray}
\end{lemma}

\begin{lemma}\label{main1}Given $\e>0$ there exist
positive constants $\dd,b,d>0$  such that 
for all large enough $n$ and  $\sigma\in \mathcal{A}_{\e}\cap \Gamma_{\dd/2}$
\begin{equation*}\mathbb{P}_{\sigma}
( \tau (\mathcal{A}^c_{\sqrt{\e}}) >e^{dn})>1-e^{-bn}.
\end{equation*}
\end{lemma}
For the definition of the sets appearing in the above statements 
see Table \ref{chart}.
The proofs of the above lemmas are significantly more delicate 
than the proofs of Lemmas \ref{closeopt1} and \ref{goodbad1}. 
We now proceed to the proofs. 
However for that we need some estimates about internal DLA on cylinder graphs.

\subsection{IDLA on the cylinder}\label{secidla}
Recall the definition of internal DLA from Subsection \ref{idla1}. In this section we discuss the case when the underlying graph is the infinite cylinder $C_n \times \mathbb{Z}_{\ge 0}$ and the starting locations of the particles are uniformly distributed on $C_n \times \{0\}$. By way of comparison with the results stated in \cite{jls1}, we will show a cruder bound on the fluctuation but we will show that the failure probability is exponentially small. We will also need to consider slightly more general starting configurations, as described below.

We now formally define a generalized IDLA process.
Consider the graph 
\begin{equation}\label{infinitecylinder}
\mathcal{C}_n=C_n \times \mathbb{Z}_{\ge 0}.
\end{equation}
For any integer $j\ge 0$ we also define the set 
\begin{equation}\label{setnotation}
Y_{j,n}=C_n \times \{j\}.
\end{equation}

\begin{definition} For integers $t\ge 0$ the sequence of random subsets (cluster) $I(t)$ is defined inductively. 
Let $I(0)$ be any arbitrary subset of $\mathcal{C}_n.$ 
Now given $I(t-1)$ start a random walk uniformly on the set $Y_{0,n}.$ $$I(t)\setminus I(t-1)$$ consists of the site at which the random walk exits $I(t-1)$ for the first time. 
\end{definition}
\begin{remark}
For the purposes of this article we will consider the initial cluster to be a union of rows, i.e. 
 $$I(0)=C_n \times \mathcal{I}_n,$$

with $\mathcal{I}_n=\{i_1,i_2,\ldots i_m\}$ 
for non-negative integers $\{i_1,i_2,\ldots i_m\}$ 
with $m$ allowed to depend on $n$.  
\end{remark}
For convenience we introduce the following notation. 
Let $$\phi:\mathbb{Z}_{\ge 0} \rightarrow 
\mathbb{Z}_{\ge 0}\setminus \mathcal{I}_n$$
be the bijection that preserves order.
Abusing notation slightly we write
$$\phi: \mathcal{C}_n \rightarrow \mathcal{C}_n \setminus I(0)$$ 
to denote the map that is identity in the first co-ordinate 
and $\phi$ in the second coordinate.
Also we define $$A(t):= \phi^{-1}(I(t)\setminus I(0)).$$
with $A(0)=\emptyset$ i.e.\ $A(t)$ is the set of new sites 
in the growth cluster under the map $\phi^{-1}$. See Figure~\ref{invmap}. 
We now state the result about the growth cluster in the generalized version 
of IDLA on the cylinder to be proved in Appendix \ref{pf}.
\begin{figure}
\centering
\includegraphics[scale=.6]{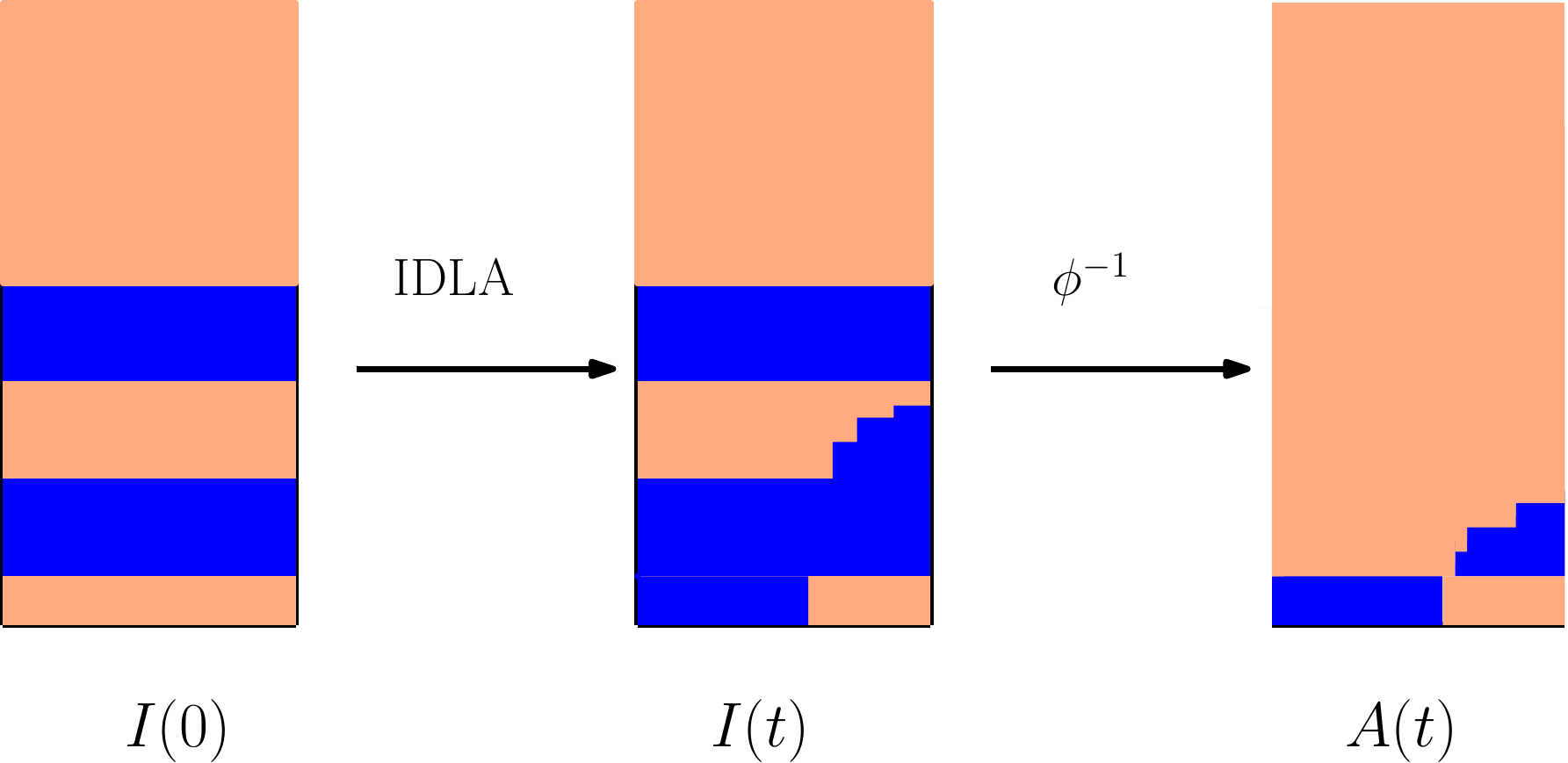}
\caption{The  IDLA cluster $I(t)$, and the cluster $A(t)= \phi^{-1}(I(t)\setminus I(0)).$ }
\label{invmap}
\end{figure}

\begin{theorem}\label{thm:IDLAcylinder} Given positive number $k<1$  for any small enough $\e$ there exists a positive number $c=c(k,\e)$ such that for all large enough $n$ and all $\mathcal{I}_n$ with $m\le n$ 
$$\mathbb{P}\Bigl(C_n \times{[0,(1-\epsilon)kn]}\subset A(kn^2)\subset C_n\times {[0,(1+\e)kn]}\Bigr)\ge 1-e^{-cn}.$$
\end{theorem}

\subsection{Proof of Lemma \ref{linehit2}}\label{seclinehit2}
We start by defining some notations and making a few remarks. 
Recall from \eqref{region} the definition of $B_{i}(\sigma)$ for $i=1,2$. 
Given $\sigma\in \Omega \cup \Omega'$ for $i=1,2,$ 
let $I_{i}(t)$ be the blue and red IDLA clusters respectively 
with initial clusters $I_{i}(0)=B_{i}(\sigma).$ 
Note that for the process $I_{2}(t)$ 
the particles start from $C_n \times \{1\}.$

Also for $i=1,2$ let $\tilde{I}_{i}(t)$ 
be variants of the processes ${I}_{i}(t)$ 
where the blue and red random walks are killed on hitting the lines 
$y=\alpha-\frac{\sqrt{\e}}{2}$ and $y=\alpha+\frac{\sqrt{\e}}{2}$ respectively.
We will refer to these as the \emph{killed IDLA processes}.
One immediately notices that starting from the same initial condition 
the trivial coupling (using the same random walks for both the processes) 
gives the inclusions 
\begin{eqnarray}
\label{idladomination1}
\tilde{I}_{i}(t)&\subseteq & I_{i}(t),\\
\label{idladomination}
B_{i}(\sigma_{t})&\subseteq & I_{i}(t)
\end{eqnarray}
for $i=1,2$ and all times $t.$
Now suppose that $\sigma^{(1)}$ and $\sigma^{(2)}$ are two initial conditions
with $B_{1}(\sigma^{(1)})\subseteq B_{1}(\sigma^{(2)})$. 
Then again under the trivial coupling 
\begin{eqnarray}\label{initdom}
{I}^{(1)}_{1}(t)&\subseteq & {I}^{(2)}_{1}(t),\\
\tilde{I}^{(1)}_{1}(t)&\subseteq & \tilde{I}^{(2)}_{1}(t)
\end{eqnarray}
where for $j=1,2,$ ${I}^{(j)}_{1}(\cdot)$ and $\tilde{I}^{(j)}_{1}(\cdot)$ 
are the blue  IDLA and killed IDLA clusters  
with initial cluster $B_{1}(\sigma^{(j)})$. 
Clearly similar statements hold for the red cluster. We omit them for brevity.

Throughout the rest of the proof we will think of the competitive erosion processes $B_{1}(\sigma_{t}),B_{2}(\sigma_{t})$ and the IDLA processes $I_{1}(\cdot),I_{2}(\cdot)$ starting from $B_{1}(\sigma_0),B_{2}(\sigma_0)$ respectively to be all coupled under the trivial coupling. For neatness  we will say that an event holds \emph{with overwhelming probability} if the probability of occurrence of the event is at least $1-e^{-cn}$ for some $c>0$ not depending on $n$.
Also let $$T=\frac{\sqrt{\e}}{10}n^2$$ where $\e$ appears in the statement of Lemma \ref{linehit2}.\\

\noindent 
\textit{Proof of Lemma \ref{linehit2}}.
By symmetry it suffices to show that at time $T$ with overwhelming probability all vertices below the line  $y<\alpha-\sqrt{\e}$ are blue. 

Define the event 
$$
\mathcal{B}_{\e}:=\left\{ B_{2}(\sigma_{t}) \mbox{ does not intersect the line } y=\alpha-\frac{\sqrt{\e}}{2}, \mbox{ for } t\le T \right\}.$$
Now given $\sigma_0=\sigma \in \Omega_{\e},$
\begin{equation}\label{idlevent}
\mathcal{B}_{\e} \mbox{ occurs with overwhelming probability. }
\end{equation}

To see this notice that there are at least $(\frac{\sqrt{\e}}{2}-\e)n$ lines 
between the levels $y=\alpha-\frac{\sqrt{\e}}{2}$ and $y=\alpha$ 
that have no red vertices. 
Thus by the upper bound in Theorem \ref{thm:IDLAcylinder} with overwhelming probability  $I_{2}(t)$ will not intersect $y=\alpha-\frac{\sqrt{\e}}{2}$ for $t\le T$. \eqref{idlevent} now follows from \eqref{idladomination}.

We now show that on the event $\mathcal{B}_\e$ with overwhelming probability at time $T$ all the initial red vertices below the line $y=\alpha-\sqrt{\e}$ in the competitive erosion process become blue, i.e. 
$$C_n \times [0,\alpha-\sqrt{\e}]\subset B_{1}(\sigma_{T}).$$ 
The strategy to show this is to compare the processes $B_{1}(\sigma_{t})$ and $\tilde{I}_{1}(t)$. 
Since on the event $\mathcal{B}_{\e}$ no new red particles hit the line $y=\alpha-\frac{\sqrt{\e}}{2}$ up to time $T$ the blue competitive erosion cluster dominates the blue killed IDLA cluster.
That is, 
\begin{equation}\label{idarg1}
\mathcal{B}_1:=\left \{\tilde{I}_{1}(t)\subset B_{1}(\sigma_{t}) \mbox{ for } t\le T \right \} \mbox{ occurs with overwhelming probability.}
\end{equation}
 Note however that the domination works in the reverse direction for the actual blue IDLA cluster by \eqref{idladomination} and hence the need for the killed process. 
We would be done at this point if the lower bound in Theorem \ref{thm:IDLAcylinder} was stated for $\tilde{I}_{1}(t)$.  However since Theorem \ref{thm:IDLAcylinder} is for $I_1(t)$ we need to argue a little more.

Now given the initial condition $\sigma$ consider a modification $\sigma'$ obtained by making all the rows between the lines $y=\alpha-\sqrt{\e}$ and line $y=\alpha-\frac{\sqrt{\e}}{2}$ completely red. Then by \eqref{initdom} for all $t$ 
\begin{equation}\label{inclusion}
\tilde{I}'_{1}(t) \subset \tilde{I}_{1}(t)
\end{equation}
where the LHS and the RHS are the killed blue IDLA clusters with initial conditions $\sigma'$ and $\sigma$ respectively.

Also define ${I}_1'(t)$ to be the blue IDLA cluster with starting condition $\sigma'.$
By the upper bound in Theorem \ref{thm:IDLAcylinder} with overwhelming probability ${I}_1'(t)$ does not intersect the line $y=\alpha-\frac{\sqrt{\e}}{2}$ for $t\le T$. 
This is because in $\sigma'$ all the $\frac{\sqrt{\e}}{2}n^2$ vertices between  $y=\alpha-\sqrt{\e}$ and line $y=\alpha-\frac{\sqrt{\e}}{2}$ are colored red. Note that this implies that both the blue IDLA and the killed IDLA clusters  starting from $\sigma'$ are exactly the same for the first $T$ steps. 
That is,
\begin{equation}\label{idarg2}
\mathcal{B}_{2}:=\left\{\tilde{{I}}_1'(t)={I}_1'(t) \mbox{ for }t\le T \right\} \mbox{ occurs with overwhelming probability.}
\end{equation}

Now by the lower bound in Theorem \ref{thm:IDLAcylinder}  
\begin{equation}\label{idarg3}
\mathcal{B}_{3}:=\left \{ C_n \times [0,\alpha-\sqrt{\e}]\subset{I}_1'(T)\right\} \mbox{ occurs with overwhelming probability }
\end{equation}
 since $\sigma'$ has at most $\e n$ lines containing red vertices below the line $\alpha-\sqrt{\e}$.
Lastly notice that \eqref{inclusion} implies 
the following containment of events: 
\begin{align*}
\mathcal{B}_{1}\cap\mathcal{B}_{2}\cap\mathcal{B}_{3}\subseteq \left\{C_n \times [0,\alpha-\sqrt{\e}]\subset B_{1}(\sigma_{T})\right\}.
\end{align*}
Thus we are done since by \eqref{idarg1},\eqref{idarg2} and \eqref{idarg3} all the event on the LHS occur with overwhelming probability.

\qed\\

Having proved Lemma \ref{linehit2} we state and prove a few preliminary lemmas required for the proof of Lemma \ref{main1}.
\begin{lemma}\label{prelimbound}For small enough $\e>0$ for any $\sigma\in \Omega_{\e}\cap\mathcal{A}_{\sqrt{\e}}$
\begin{equation}\label{eqprelim}
\mathbb{P}_{\sigma}(\tau({\Omega\setminus \mathcal{A}_{\sqrt{\e}}})\ge \e n^2)\ge 1-e^{-cn}
\end{equation}
for some $c=c(\e)>0$. For the definitions of  $\mathcal{A}_\e$  and $\Omega_{\e}$ refer to Table \ref{chart}.
\end{lemma}
\begin{proof}

As $\sigma\in \Omega_{\e}$ there are at least $(\sqrt{\e}-\e)n$ rows with no red vertex between the lines $y=\alpha-\sqrt{\e}$ and $\alpha$ and similarly at least $(\sqrt{\e}-\e)n$ rows with no blue vertex between the lines $y=\alpha$ and $\alpha+\sqrt{\e}.$
The proof now follows from the upper bound in Theorem \ref{thm:IDLAcylinder} 
which implies that the IDLA processes $I_{1}(t),I_{2}(t)$ require at least $\e n^2$ rounds with probability at least $1-e^{-cn}$ for some positive $c=c(\e)$ before a blue vertex is seen above the line $y=\alpha+\sqrt{\e}$ and a red vertex is seen below the line $y=\alpha-\sqrt{\e}$ or a . Thus we are done by \eqref{idladomination} which says that the same holds for the competitive erosion clusters $B_{i}(t)$ for $i=1,2.$
\end{proof}

For any set $A\subset \Omega$ define the positive return time $\tau^{+}(A)$ to be $$\inf\{t\ge 1 : \sigma_t \in A\}.$$

\begin{lemma}\label{hitlemma1}Given positive numbers $\e,c$ there exist positive constants $\dd,D$ such that for any $\sigma\in \Omega_{\e}\cap \Gamma_{\dd}$
\begin{equation}\mathbb{P}_{\sigma}(\tau^+(\Omega_{\e})\ge c n^2)<e^{-Dn^2}.
\end{equation}
\end{lemma}

\noindent
For the definition of the sets see Table \ref{chart}.
\begin{proof} 

The proof follows from Lemma \ref{lemm:azuma1} $i.$ 
First let $\sigma'=\sigma_1$ for 
the competitive erosion chain starting from $\sigma_0=\sigma$. 
If $\sigma' \in \Omega_{\e}$ we are done. 
Otherwise by \eqref{gradbound} and the fact that by hypothesis $\sigma \in \Gamma_{\dd}$,
$$x:=h(\sigma')\ge \alpha(1-\alpha/2)n^2-\dd n^2 .$$

Now let $$X(t)=h(\sigma_{t+1})$$ 
where $h(\cdot)$ is the height function defined in \eqref{wf}. 
Thus $X(0)=x.$
We make the following choice of parameters:
\begin{align*}
B &= \Omega_{\e}\\
A_1 &= \alpha(1-\alpha/2)n^2 \\
A_2 &= 1\\
a_1 &= a(\e)\mbox{ appearing in Theorem }\ref{nonnegativedrift}\\
a_2 & = \dd n^2 \mbox{ with }\dd < ac/2  \\
T &= cn^2.
\end{align*}The choice of $a_2$ works since $\sigma\in \Gamma_{\dd}$ by hypothesis.
Also the choice of $\dd$ ensures that $a_2-a_1T <0$ and hence the hypothesis of Lemma \ref{lemm:azuma1} $i.$ is satisfied. 
Thus for $\dd=\frac{ac}{2}$ by Lemma \ref{lemm:azuma1} $i.$ for all $\sigma \in \Omega_{\e}\cap \Gamma_{\dd},$ 
$$\mathbb{P}_{\sigma'}(\tau(\Omega_{\e})\ge T)\le \exp\left({-\frac{a^2_1c n^2}{16}}\right).$$ 
Since $\tau^{+}(\Omega_{\e})$ starting from $\sigma$ is one more than $\tau(\Omega_{\e})$ starting from $\sigma'$ we are done.
\end{proof}

\begin{lemma}\label{finalargument1}Given $\e>0$ there exist positive constants $\dd,b$ such that for large enough $n$ and any $\sigma\in \Omega_{\e}\cap \Gamma_{\dd}\cap \mathcal{A}_{\sqrt{\e}},$
\begin{equation*}\mathbb{P}_{\sigma}(\tau^+(\Omega_{\e})\le \tau(\Omega \setminus \mathcal{A}_{\sqrt{\e}}) )>1-e^{-bn}.
\end{equation*}
\end{lemma}

\begin{proof} Since $\sigma \in \Omega_{\e}\cap \mathcal{A}_{\sqrt{\e}}$ by Lemma \ref{prelimbound}
\begin{equation*}
\mathbb{P}_{\sigma}(\tau(\Omega\setminus \mathcal{A}_{\sqrt{\e}})\ge \e n^2)\ge 1-e^{-hn}
\end{equation*}
for some $h=h(\e)>0.$
Now using $c=\e/2$ in Lemma \ref{hitlemma1}
we get that we can choose $\dd$ such that for any $\sigma\in \Omega_{\e}\cap \Gamma_{\dd}\cap \mathcal{A}_{\sqrt{\e}}$
\begin{equation}\label{quadreturn}
\mathbb{P}_{\sigma}(\tau^+(\Omega_{\e})\ge \e n^2/2)<e^{-Dn^2}.
\end{equation}
for some positive constant $D$.
Thus for such a $\dd$ 
\begin{eqnarray*}
\mathbb{P}_{\sigma}(\tau^+(\Omega_{\e})\le \tau(\Omega \setminus \mathcal{A}_{\sqrt{\e}}) )& \ge & \mathbb{P}_{\sigma}(\tau^+(\Omega_{\e})\le \e/2 n^2)-\mathbb{P}_{\sigma}(\tau(\Omega\setminus \mathcal{A}_{\sqrt{\e}})\le \e n^2)\\
 &\ge & 1-e^{-hn}-e^{-Dn^2}.
\end{eqnarray*}
Hence we are done by choosing $b=\frac{h}{2}$.
\end{proof}

We are now ready to prove Lemma \ref{main1}.
\subsection{Proof of Lemma \ref{main1}}\label{secmain1}
We will specify $\dd$ later. However notice that for any small enough $\dd$ if
 $\sigma\in \Gamma_{\dd/2}$ then by Lemma \ref{goodbad1} there exist  positive constants $c,d$ depending on $\dd$ such that for large enough $n$ 
\begin{equation}\label{ht33}
\mathbb{P}_{\sigma}(\tau' >e^{cn^2})  \ge  1- e^{-{d}n^2}
\end{equation}
where $\tau'=\tau(\Omega \setminus \Gamma_{\dd}).$ 
Now notice that by hypothesis $\sigma\in \mathcal{A}_{\e}$ 
and hence $\sigma \in \Omega_{\e}.$
Let $\tau_{(1)},\tau_{(2)}\ldots$ 
be successive return times to $\Omega_{\e}$, 
i.e.\ $\tau_{(1)}=0$ and for all $i\ge 0$, 
$$\tau_{(i+1)}=\inf\{t:t>\tau_{(i)},\sigma_{t} \in \Omega_{\e}\}.$$ 
 
The following containment is true for any positive $b'$: Let $$s:=e^{b'n}.$$
Then 
$$\{\tau(A^c_{\sqrt{\e}})\le s\}\subset \left\{\exists \,i\le s \mbox{ such that }\sigma_{i}\notin \Gamma_{\dd}\right\}\cup\left(\bigcup_{j=0}^{s}\left\{\bigl\{\tau_{(j)}< \tau(A^c_{\sqrt{\e}})\le \tau_{(j+1)}\bigr\}\cap \bigl\{ \sigma_{\tau_{(j)}}\in \Gamma_{\dd}\bigr\}\right\}\right).$$
The above follows by first observing whether there exists  any $i<e^{b'n}$ such that $\sigma_{i}\notin \Gamma_{\dd}$. Also let $j$ be the first index such that  $$\tau_{(j)}< \tau({A^c_{\sqrt{\e}}})\le \tau_{(j+1)}.$$

Notice that on the event $\{\tau(A^c_{\sqrt{\e}})\le s\}$, $$j\le s.$$

Also by definition on the event $\bigl\{\sigma_{i}\in \Gamma_{\dd} \mbox{ for all }  i \le s\bigr\}$, $$\sigma_{\tau_{(j)}}\in \Gamma_{\dd}\cap \Omega_{\e}\cap \mathcal{A}_{\sqrt{\e}}.$$
Thus by the union bound,
\begin{eqnarray*}\mathbb{P}\bigl(\tau(A^c_{\sqrt{\e}})\le e^{b'n}\bigr)&\le & \mathbb{P}(\exists \,i\le s \mbox{ such that }\sigma_{i}\notin \Gamma_{\dd})\\ &+&   \sum_{j=1}^{e^{b'n}}\mathbb{P}\left(\bigl\{\sigma_{\tau_{(j)}}\in\Gamma_{\dd}\cap \Omega_{\e}\cap \mathcal{A}_{\sqrt{\e}}\bigr\}\cap \bigl\{\tau(A^c_{\sqrt{\e}})\le \tau_{(j+1)}\bigr\}\right).
\end{eqnarray*}
Recall that by hypothesis $\sigma \in \Gamma_{\dd/2}.$ Thus by \eqref{ht33} for any small enough $\dd$ there exists $d=d(\dd)>0$ such that for any $b'>0$ the first term is less than $e^{-dn^2}$ for large enough $n$.
Also notice that by Lemma \ref{finalargument1} we can choose $\dd$ such that every term in the sum is at most $e^{-bn}$ for some constant $b>0$. 
Thus for such a $\dd$, putting everything together we get that for any $b'>0$ and large enough $n,$ 
$$\mathbb{P}(\tau(A^c_{\sqrt{\e}}) \le e^{b'n}) \le e^{-dn^2}+\sum_{i=1}^{e^{b'n}}e^{-bn}.$$

Hence by choosing $b'<b$ we get that for large enough $n,$ $$\mathbb{P}(\tau(A^c_{\sqrt{\e}})\le e^{b'n})\le 2e^{(b'-b)n}.$$
The proof is thus complete.
\qed

\subsection{Proof of Theorem \ref{t.quickhit}}\label{pom2} 
We will use Lemma \ref{hitstation} and prove something stronger which will be used in the proof of Theorem \ref{mainresult}. However for brevity we first need some notation.
We start by recalling the sets in  Table \ref{chart}. Now given any positive $\e_1$ and $\e_2$ 
define the set $$\mathcal{C}_{\e_1,\e_2}=\mathcal{A}_{\e_1}\cap \Gamma_{\e_2}.$$ 
We claim that for small enough $\e,\dd>0$ 
there exists  constants $D,D'>0$ such that for any $\sigma\in \Omega,$
\begin{equation}\label{hittingestimate1}
\mathbb{P}_{\sigma}\bigl(\tau(\mathcal{C}_{\sqrt{\e},\dd/2})\ge Dn^2 \bigr)
\le e^{-D'n}.
\end{equation}
Clearly this proves Theorem \ref{t.quickhit} since $\mathcal{C}_{\sqrt{\e},\dd/2} \subset \mathcal{A}_{\sqrt{\e}}$.
Let
\begin{eqnarray*}
\tau' &=& \tau (\Gamma_{\dd/4}) \\
\tau'' &=& \tau({\Gamma_{\dd/2}^\mathsf{c}}) \\
\tau''' &=& \tau({\Omega_{\e}}) \\
\tau'''' &=& \tau({\mathcal{A}_{\sqrt{\e}}}).
\end{eqnarray*}
Recall the notational convention made in Subsection \ref{nc}.
Now to show \eqref{hittingestimate1} 
we notice the following containment of events:
$$\left[\{\tau'\le An^2\} \cap\{\tau'' \ge e^{cn}\mid{\sigma_{\tau'}}\}
\cap \{\tau''' \le An^2 \mid{\sigma_{\tau'}}\}
\cap \{\tau'''' \le An^2 \mid{\sigma_{\tau'''}}\}\right]
\subset \{\tau({\mathcal{C}_{\sqrt{\e},\dd/2}})\le 3An^2\}.$$ 
To see why this containment holds we first notice that
$$\Gamma_{\dd/4}\subset \Gamma_{\dd/2}.$$
Thus the first two events imply that the process 
hits the set $\Gamma_{\dd/2}$ in $An^2$ steps 
and stays inside for an exponential (in $n$) amount of time, 
and in particular stays in $\Gamma_{\dd/2}$ 
from time $An^2$ through time $3An^2$. 
In addition, the third and fourth events together imply that 
regardless of where the chain is at time $An^2$,
the chain enters $\Omega_{\e}$ by time $2An^2$
and then enters $\mathcal{A}_{\sqrt{\e}}$ by time $3An^2$.
Hence in the intersection of the four events,
the hitting time of 
$\mathcal{C}_{\sqrt{\e},\dd/2}=\mathcal{A}_{\sqrt{\e}}\cap \Gamma_{\dd/2}$ 
is at most $3An^2$.  Let $D=3A$.
Now for a large enough constant $A$,  there exists  $D'>0$ such that the probabilities of 
 of all the four events on the left hand side are at least $1-e^{-D'n}$. This follows by Lemmas \ref{closeopt1}, \ref{goodbad1}, \ref{linehit} and \ref{linehit2}  respectively.
Hence by the union bound \eqref{hittingestimate1} follows for a slightly smaller value of $D'$. 
\qed 
\subsection{Proof of Theorem \ref{mainresult}}\label{pom1}
Note that \eqref{hittingestimate1} is true for all small enough choices of $\e$ and $\dd.$ 
However given $\e$ by Lemma \ref{main1} there exist $\dd,b,d>0$ such that for $\sigma \in \mathcal{C}_{\sqrt{\e},\dd/2}$
\begin{equation}\label{last1}\mathbb{P}_{\sigma}
( \tau (\mathcal{A}^{c}_{\e^{1/4}}) >e^{dn})>1-e^{-bn}.
\end{equation}

For such a $\dd$ the proof of Theorem \ref{mainresult} follows by using Lemma \ref{hitstation} with the following choices of parameters:
\begin{eqnarray*}
A&=& \mathcal{C}_{\sqrt{\e},\dd/2}\\
B&=& \mathcal{A}_{{\e}^{1/4}}\\
t_1&=&Dn^2\\
t_2&=&e^{dn}\\
p_1 &=& e^{-D'n}\\
p_2 &=& e^{-bn}.
\end{eqnarray*}
The above choice of parameters satisfy the hypotheses of Lemma \ref{hitstation} by  \eqref{hittingestimate1} and \eqref{last1}.
\qed

\subsection{Proof of Theorem \ref{diffthm1} }\label{pom3}
The proof is a simple corollary of Theorem \ref{mainresult}. Let $k=\frac{2}{\e}.$ For $i=1\ldots k$, let $$\alpha_i=\frac{\e}{2}i.$$ Now define $$\Lambda_{i}=\{(x,y)\in \mathcal{C}_n:\lambda(x,y)\le \alpha_i n^2\}.$$
Choose $\dd=\e^{100}.$ By Theorem \ref{mainresult} for all $i=1,2,\ldots k$ there exists $\beta_i$ such that 
\begin{align}\label{event12}
{\hat{\pi}}_n\left(\forall (x,y)\in \Lambda_i : y\le \alpha_{i}+\dd\right )&\ge 1- e^{-\beta_i n}\\
\nonumber
{\hat{\pi}}_n\left(\forall (x,y)\notin \Lambda_i : y\ge \alpha_{i}-\dd\right.)&\ge 1-e^{-\beta_i n}.
\end{align}
By union bound there exists $\beta$ such that the events on the LHS of \eqref{event12} simultaneously hold for all $i$ with failure probability at most $e^{-\beta n}.$ Clearly this completes the proof. To see this suppose that all the events hold and there exists $(x,y) \in \mathcal{C}_n$ such that $\frac{\lambda(x,y)}{n^2}-y \ge \e.$ Let $1\le j\le k$ be such that $$\alpha_{j-1}n^2 <\lambda(x,y)\le \alpha_{j}n^2$$ (take $\alpha_0=0$). Now by the second event in \eqref{event12} $y\ge \alpha_{j-1}-\dd$ which implies that
$$\frac{\lambda(x,y)}{n^2}-y \le \alpha_{j}-\alpha_{j-1}+\dd =\frac{\e}{2}+\e^{100}< \e$$ which is a contradiction.
Similarly one can show $\frac{\lambda(x,y)}{n^2}-y >-\e.$ Details are omitted.  \qed.

\section*{Appendix}
\appendix
\section{IDLA on the cylinder}\label{pf}
The proof of  Theorem \ref{thm:IDLAcylinder} follows by adapting the ideas of the proof appearing in \cite{lbg}. The proof in  \cite{lbg} follows from a series of lemmas which we now state in our setting. Recall \eqref{setnotation}.
 Let $\tau_{z}$, $\tilde{\tau}_{kn}$ be the hitting times of  $\phi(z)$ and $\phi(Y_{kn,n})$ respectively.   
\begin{lemma}\label{avmore}For any $z=(x,y)\in \mathcal{C}_n$ with $y\le kn$
 
$$kn^2\mathbb{P}_{Y_{0,n}}(\tau_{z}<\tilde \tau_{kn}) \ge \sum_{w\in C_n\times[0,kn)} \mathbb{P}_{\phi(w)}(\tau_{z}<\tilde \tau_{kn})$$
where $\mathbb{P}_{\phi(w)}$ and $\mathbb{P}_{Y_{0,n}}$ are the random walk measures on $\mathcal{C}_n$ with starting point $\phi(w)$ and uniform over $Y_{0,n}$ respectively.
\end{lemma}

\begin{proof}   By symmetry in the first coordinate under $\mathbb{P}_{Y_{0,n}}$ for any $j$ the distribution of the random walk  when it hits the set $Y_{j,n}$ is uniform over the set $Y_{j,n}$. Hence by the Markov property the chance that random walk hits $\phi(z)$ before $Y_{\phi(kn),n}$ after reaching the line $Y_{\phi(j),n}$ is $$\frac{1}{n}\sum_{w\in Y_{\phi(j),n}}\mathbb{P}_{w}(\tau_{z}<\tilde \tau_{kn}).$$ Thus clearly for any $j< kn$
\begin{equation}\label{hitsum}
\mathbb{P}_{Y_{0,n}}( \tau_{z}<\tilde \tau_{kn})\ge \frac{1}{n}\sum_{w\in Y_{\phi(j),n}}\mathbb{P}_{w}(\tau_{z}< \tilde \tau_{kn}).
\end{equation}
The lemma follows by summing over $j$ from $0$ through $kn-1$.
\end{proof}
\begin{lemma}\label{greensfunctionlower}Given positive numbers $k$ and $\e$ with $\e< 1$, then there exists $\beta=\beta(k,\e)$ such that for all $z=(x,y)$ with $y\le (1-\e)kn$  
$$\mathbb{P}_{Y_{0,n}}(\tau_{z}<\tilde \tau_{kn,n})\ge \frac{\beta }{\ln n}.$$
\end{lemma}
\begin{proof} Since $\phi(kn)-\phi(y)\ge \e kn$ the semi-disc of radius $\min(\e k,1/2)n$ around $z$ lies below the line $Y_{\phi(kn),n}$. The random walk starting uniformly on $Y_{o,n}$ hits the interval $(z-\min(\e k/2 ,1/4)n,z+\min(\e k/2 ,1/4)n)$ with probability at least $\min(\e k,1/2)$. Now the lemma follows by the standard result that the random walk starting within radius $n/2$ has $\Omega(\frac{1}{\log n})$ chance of returning to the origin before exiting the ball of radius $n$ in $\mathbb{Z}^2.$ This fact can be found in \cite[Prop 1.6.7]{mrw}.
\end{proof}

The next result is the standard Azuma-Hoeffding inequality 
stated for sums of indicator variables.
\begin{lemma}\label{ldp} For any positive integer $n$ 
if $X_i$ $i=1,2 \ldots n$ are independent indicator variables then
$$\mathbb{P}(|\sum_{i=1}^{n}X_i-\mu|\ge t)\le 2 e^{-\frac{t^2}{4n}}$$
where $\displaystyle{\mu=\E\sum_{i=1}^{n}X_i}.$
\end{lemma}

\subsection{Hitting estimates}\label{seche}

Consider the simple random walk $(Y(t))_{t \geq 0}$ on $\Z^2$.  

\begin{lemma}
\label{l.discretecauchy}
For $(x,y) \in \Z^2$, let $h(x,y) = \PP_{(x,y)} \{ Y(\tau (\Z \times \{0\})) = (0,0) \}$ be the probability of first hitting the $x$-axis at the origin. Then
	\[ h(x,y) = \frac{y}{\pi(x^2+y^2)} + O( \frac{1}{x^2+y^2} ). \]
\end{lemma}

\begin{proof}
Let
	\[ \widetilde h(x,y) = \begin{cases} h(x,y), & y>0 \\
							0, &y=0 \\
						     -h(x,y), &y<0. \end{cases} \]
The discrete Laplacian
	\[ \Delta \widetilde h(x,y) = \widetilde h(x,y) - \frac{\widetilde h(x+1,y)+ \widetilde h(x-1,y)+ \widetilde h(x,y+1) + \widetilde h(x,y-1)}{4} \]
vanishes except when $(x,y) = (0,\pm 1)$, and $\Delta \widetilde h(0,\pm 1) = \pm \frac14$.  Since $\widetilde h$ vanishes at $\infty$ it follows that
	\begin{equation} \label{e.discretederiv} \widetilde h(x,y) = \frac{a(x,y+1) - a(x,y-1)}{4} \end{equation}
where 
	\[ a(x,y) = \frac{1}{\pi} \log(x^2+y^2) + \kappa + O(\frac{1}{x^2+y^2}) \]
is the recurrent potential kernel for $\Z^2$ (see \cite{kozma}).  Here $\kappa$ is a constant whose value is irrelevant because it cancels in the difference \eqref{e.discretederiv}.
	\end{proof}

Let $X(\cdot)$ be the simple symmetric random walk on the half-infinite cylinder $\mathcal{C}_n=C_n \times \mathbb{Z}_{\ge 0}.$
\begin{lemma}\label{he} For any positive integers $j<k$, 
with $\Delta=k-j<n $:\\\\
$i.$ for any $w \in Y_{k,n}$,
$$\mathbb{P}_{w}(\tau(j)<\tau^{+}(k))<\frac{1}{\Delta}$$
where $\tau(j)$ and $\tau^{+}(k)$ are the hitting and positive hitting times of $Y_{j,n}$ and $Y_{k,n}$ respectively for $X(\cdot)$.\\\\
$ii.$ there exists a constant $J$ such that for $w\in Y_{j,n}$ and any subset $B\subset Y_{k,n}$,
$$\mathbb{P}_w(X({\tau(k)})\in B)<J|B|/\Delta.$$
\end{lemma}
\begin{proof}
$i.$ is the following standard result about one-dimensional random walk: 
starting from $1$ the probability of hitting $\Delta$ before $0$ 
is $\frac{1}{\Delta}.$ 
\\
\noindent
Now we prove $ii.$  Clearly it suffices to prove it in the case when $B$ consists of a single element.   It is easy to check that if $Y(t)=(Y_1(t),Y_2(t))$ is the simple random walk on $\mathbb{Z}^2$,
then
\begin{equation}\label{proj12}
X(t)=(Y_1(t)\text{ mod } n,|Y_2(t)|)
\end{equation} is distributed as the simple random walk on $\mathcal{C}_n.$
For any $\ell\in \mathbb{Z}$ let $\tau_1(\ell)$ be the hitting time of the line $y=\ell,$ for $Y(t)$. 
Clearly by \eqref{proj12} for  $w=(0,j),z=(z_1,k)\in \mathcal{C}_n,$
$$\mathbb{P}_w(X({\tau(k)})=z)=\sum_{i=-\infty}^{\infty}P_{(0,j)}\left\{Y(\tau_1(k)\wedge \tau_1(-k))\in\{(z_1+in,k),(z_1+in,-k)\}\right\}.$$
 By union bound the RHS is at most 
\begin{align*}
\sum_{i=-\infty}^{\infty}P_{(0,j)}\left\{Y(\tau_1(k))=(z_1+in,k)\right\} 
+\sum_{i=-\infty}^{\infty}P_{(0,j)}\left\{Y(\tau_1(-k))=(z_1+in,-k)\right\}.
\end{align*}
Using the notation in Lemma \ref{l.discretecauchy} we can write the above sum as 
$$
\sum_{i=-\infty}^{\infty}[h(-z_1+in,\Delta)+h(-z_1+in,k+j)]
$$
By Lemma \ref{l.discretecauchy} the above sum is $$O(\frac{1}{\Delta})+O(\frac{1}{k+j})=O(\frac{1}{\Delta}).$$
Hence we are done.
\end{proof}

\subsection{Proof of Theorem \ref{thm:IDLAcylinder}}\label{secidlaproof}
Equipped with the results in the previous subsection the proof of Theorem \ref{thm:IDLAcylinder} will now be completed by following the steps in \cite{lbg} .\\\\
\textbf{Lower bound}: 
It suffices to show $C_n\times {[0,(1-\epsilon)kn]}\subset A({(1+\epsilon)kn^2})$. Fix $z\in C_n\times {(1-\epsilon)kn}$ For any positive integer $i$ we associate the following stopping times to the $i^{th}$ walker:  
\begin{itemize}
\item $\sigma^i$: the stopping time in the IDLA process 
\item $\tau^{i}_z$: the hitting time of $\phi(z)$
\item $\tau^{i}_{kn,n}$:  the hitting time of the set $\phi(Y_{kn,n}).$
\end{itemize}   
Now we define the random variables
\begin{eqnarray*}
N&=&\sum_{i=1}^{{(1+\epsilon)kn^2}}\mathbf{1}_{(\tau^{i}_z<\sigma^{i})}\mbox{, the number of particles that visit $\phi(z)$ before stopping}\\
M&=&\sum_{i=1}^{{(1+\epsilon)kn^2}}\mathbf{1}_{(\tau^{i}_z<\tau^{i}_{kn,n})}\mbox{, the number of particles that visit $\phi(z)$ before reaching $\phi(Y_{kn,n})$}\\
L&=&\sum_{i=1}^{{(1+\epsilon)kn^2}}\mathbf{1}_{(\sigma^{i}<\tau^{i}_z <\tau^{i}_{kn,n})}  \mbox{, the number of particles that visit $\phi(z)$ before reaching $\phi(Y_{kn,n})$}\\
& &\mbox{ but after stopping}.
\end{eqnarray*}
Thus $$N\ge M-L.$$
Hence 
\begin{equation}\label{keybound}
\mathbb{P}\Bigl(z\notin A((1+\epsilon)kn^2)\Bigr)=\mathbb{P}(N=0)\le \mathbb{P}(M<a)+\mathbb{P}(L>a).
\end{equation}
where the last inequality holds for any $a.$
Now by definition  
\begin{equation}\label{inter1}
\E(M)=(1+\epsilon)kn^2\:\mathbb{P}_{Y_{0,n}}(\tau_{z}<\tau_{kn,n}).
\end{equation}
We now bound the expectation of $L$. Define the following quantity: let independent random walks start from each $w\in {\phi(C_n\times [0,kn])}$ and let $$\tilde{L}=\sum_{w\in \phi(C_n\times [0,kn])}\mathbf{1}{(\tau_{z}<\tau_{kn,n} \mbox{ for the walker starting at } w)}.$$
Clearly $L\le \tilde{L}.$ Hence the RHS of \eqref{keybound} can be upper bounded by $\mathbb{P}(M<a)+\mathbb{P}(\tilde{L}>a)$.
Now
$$\E(\tilde{L})= \sum_{w\in \phi(C_n\times [0,kn])} \mathbb{P}_w(\tau_{z}<\tau_{kn,n}).$$
Also by Lemma \ref{avmore} and \eqref{inter1} $$ \Bigl(1+\frac{\e}{2}\Bigr)\E(\tilde{L})\le \E(M).$$

Choose $a=(1+\epsilon/4)\max \bigl(\frac{\beta kn^2}{\ln n},\E(\tilde{L})\bigr)$ where the $\beta$ appears in Lemma \ref{greensfunctionlower}. Now using Lemma \ref{ldp} we get
\begin{eqnarray*}
\mathbb{P}(\tilde{L}>a)& \le & \exp(-dn)\\
\mathbb{P}(M<a)& \le & \exp(-dn)
\end{eqnarray*}
for some constant $d=d(\e,k)>0.$
Thus in \eqref{keybound} we get
$$\mathbb{P}(M<a)+\mathbb{P}(L>a)\le 2\exp(-dn).$$ 
The proof of the lower bound now follows by taking the union bound:
\begin{eqnarray*}
\mathbb{P}(C_n\times {[0,(1-\epsilon)kn]} \subset A((1+\epsilon)kn^2)&\le & \sum_{z\in C_n\times {[0,(1-\epsilon)kn]}} \mathbb{P}\Bigl(z\notin A((1+\epsilon)kn^2)\Bigr)\\
& \le & \sum_{z\in C_n\times {[0,(1-\epsilon)kn]}}2\exp(-dn)\\
&\le & \exp(-cn).
\end{eqnarray*}
where the last inequality holds for large enough $n$ when $c$ is smaller than $d$. \\\\
\textbf{Upper bound}: In \cite{lbg} the upper bound is proven by showing that the growth of the cluster above level $(1+\epsilon)kn$ is dominated by a multitype branching process. However here we slightly modify the proof to take into account that in our situation the initial cluster is not empty. We define some notation.
Let us denote the particles making it out of level $\phi(Y_{kn,n})$ by $w_1,w_2,\ldots$ and define $$\tilde{A}(j)=A(w_j).$$
Choose $k_0=k(1+\sqrt{\e})n.$ 
We define 
\begin{equation}\label{moddef}
\tilde{Y}_{\ell,n}:=Y_{k_0+\ell,n}.
\end{equation}
Given the above notation let $$Z_{\ell}(j):=\tilde{A}(j)\cap \tilde{Y}_{\ell,n}.$$
Define $$\mu_{\ell}(j):=\E(Z_{\ell}(j)).$$
\begin{lemma}\cite[Lemma 7]{lbg}\label{lemma:multitype} There exists a universal $J_1>0$ such that for all $k, \e\in(0,1),$ $n\ge N(k,\e)$ and all positive integers $j,\ell$  
$$\mu_{\ell}(j)<kn\left(J_1 \frac{j}{\ell} \frac{1}{\sqrt{\epsilon} kn}\right)^{\ell}.$$
\end{lemma}
We include the proof of Lemma \ref{lemma:multitype} for completeness. However first we show how it implies the upper bound in Theorem \ref{thm:IDLAcylinder}. 
Let us define the event $$F:=C_n\times {(1-\epsilon)kn}\subset A(kn^2).$$

Now let $B=B(k)>0$ be a constant to be specified later.
Then
\begin{eqnarray*}
\mathbb{P}\bigl(A(kn^2)\not\subset C_n\times[0,k(1+B\sqrt{\epsilon})n]\cap F\bigr)& \le & \mathcal \mathbb{P}(Z_{n' }(2k \epsilon n^2)>1)\\
&\le & \mu_{n'}(2k\epsilon n^2)
\end{eqnarray*} 
where $n'=(k B)\sqrt{\epsilon}n-1 -k\sqrt{\e} n.$
To see why these inequalities are true first note that the set $\tilde{Y}_{n',n}$ is at height less than $k(1+B\sqrt{\epsilon}n).$
Hence the cluster at time $kn^2$ should intersect $\tilde{Y}_{n',n}$ to grow beyond height $k(1+B\sqrt{\epsilon})n.$ However on the event $F$ at most $2\e k n^2$ particles out of the first $kn^2$ move beyond height $kn.$ Hence the size of the intersection of $\tilde{Y}_{n',n}$ and the cluster is at most $Z_{n' }(2k \epsilon n^2).$ Thus we get the first inequality. The  second inequality follows trivially from the fact that for a non-negative integer-valued random variable the expectation is at least as big as the probability of the random variable being positive. 
Using Lemma \ref{lemma:multitype} we get 
\begin{eqnarray*}
\mu_{n'}(2k\epsilon n^2) &\le & kn\left(J_1 \frac{2k\e n^2}{n'} \frac{1}{\sqrt{\epsilon} kn}\right)^{n'}\\
&=& kn\left(J_1 \frac{4k\epsilon n^2}{k(B-1)\sqrt{\epsilon} n} \frac{1}{\sqrt{\epsilon} kn}\right)^{k(B-1)\sqrt{\epsilon} n}\\
& = & kn\left(\frac{4J_1}{(B-1)k} \right)^{k(B-1)\sqrt{\epsilon} n}.
\end{eqnarray*}
Thus
$$\mathbb{P}\bigl(A_{kn^2}\notin [0,n]\times [0,k(1+B\sqrt{\epsilon})n]\cap F\bigr)\le kn\left(\frac{4 J_1}{(B-1)k}\right)^{(B-1)\sqrt{\epsilon}kn}.$$
Now by choosing $B$ such that $4J_1 <(B-1)k$ we are done.
\hspace{.3in}
\qed\\\\
\textbf{Proof of Lemma \ref{lemma:multitype}}.
The rate at which $\tilde{Y}_{\ell,n}$ grows is at most the rate at which a particle exiting height $kn$ reaches the occupied sites in $\tilde{Y}_{\ell-1,n}$.
Thus if $X(t)$ is the random walk on $C_n$ defined in \eqref{infinitecylinder} then for any $m$
\begin{eqnarray}\label{keyrestriction}
\mu_{\ell}(m+1)-\mu_{\ell}(m)&\le & \sup_{y\in Y_{kn,n}} \mathbb{P}_{y}\Bigl[X(\tau_{\tilde{Y}_{\ell-1,n}})\in \tilde{A}(m)\Bigr]\\
&\le & J \frac{\mu_{\ell-1}(m)}{kn\sqrt{\e}}
\end{eqnarray}
where the second inequality follows by Lemma \ref{he} $ii$.
Summing over $m=0,1\ldots j$ we get 
$$\mu_{\ell}(j)\le \frac{J}{kn\sqrt{\e}}\sum_{m=0}^{j-1} \mu_{\ell-1}(m).$$
Iterating the above relation in $\ell$ with fixed $j$ gives us 
$$\mu_{\ell}(j)\le 
\Bigl( J\frac{1}{\sqrt{\e} {kn}}\Bigr)^{\ell-1} \frac{j^{\ell}}{\ell!}.$$
The lemma follows by using the inequality $$\ell!\ge \ell^\ell e^{-\ell}.$$
\qed 

\section{Green's function and flows} \label{app2}
We prove Lemma \ref{heightgreen1}. 
We start by discussing some properties 
of the ordinary random walk on $\Cyl_n$ (defined in \eqref{graphrep}).
For any $v \in \Cyl_n$ define 
\begin{equation}\label{gf}
G_{n}(v)=\frac{1}{4n}
\E_{w} \bigl[\# \mbox{ visits to the line } y=0 \mbox{ before } \tau(C_{n} \times \{1\})   \bigr].
\end{equation}
\begin{lemma}\label{linear}  For any point $(x,y)\in \Cyl_n$ 
$$G_{n}(x,y)=1-y.$$
\end{lemma}
\begin{proof}
Consider the lazy symmetric random walk on the interval $[0,n]$ 
where at $0$ the chance that it moves to $1$ is $\frac{1}{4}$ 
and everywhere else the chance that it jumps is $\frac{1}{2}.$
By symmetry of $\Cyl_n$ in the $x$-coordinate 
it is clear that for all $(x,y)\in \Cyl_n,$ $4n G_{n}(x,y)$ 
is the expected number of times 
that the above one-dimensional random walk starting from $ny$
hits $0$ before hitting $n$. 
The above quantity is easy to compute and is $4n(1-y).$ 
\end{proof}
\begin{remark}
Thus for any $\sigma\in \Omega\cup \Omega'_{n}$ 
\begin{equation}\label{wf1}
h(\sigma)=\sum_{(x,y)\in B_1}G_n(x,y)
\end{equation}
where $h(\cdot)$ is defined in \eqref{wf}.
\end{remark}
We now define the stopped Green's function. For any $A\subset \Cyl_n$ and $v\in \Cyl_n$ define
\begin{equation}\label{gfa}
G_{A}(v)=
\frac{1}{4n}\E_{v} \bigl[\# \mbox{ visits to the line } y=0 
\mbox{ before } \tau(A^c)   \bigr].
\end{equation}

\begin{lemma}\label{stopgre} 
Given $A \subset \Cyl_n$ such that $A\cap (C_n \times \{1\})=\emptyset,$ 
for all $(x,y)$ in $\Cyl_n$ we have 
$${G}_{A}(x,y)={H}_{A}(x,y)-y$$
where  $H_{A}(\cdot)$ was defined in \eqref{stoppedheightgeneral}. 
\end{lemma}
\begin{proof} Let $y_{t}$ be the height of the walk at time $t\le {\tau}(A^c)$.
Consider the following telescopic series:
\begin{equation}
\label{telescope}
y_{\tau(A^c)}-y_{0}=
\sum_{t=0}^{\infty}(y_{t+1}-y_{t})\mathbf{1}(\tau(A^c)>t).
\end{equation}
Notice that since $A\cap (C_n \times \{1\}) =\emptyset$, 
$t<\tau_{A^c}$ implies $y_t<1$.
We make the following simple observation:
\begin{eqnarray}\label{laplc}
\E\bigl[(y_{t+1}-y_{t})\mathbf{1}
({\tau}(A^c)>t)|\mathcal{F}_t \bigr] & = & \left\{
\begin{array}{ccc}
 \frac{1}{4n} \mathbf{1}({\tau}({A^c})>t) &\mbox{}  y_t=0  \\
 0           &\mbox{}  y_t>0  
\end{array}
\right. 
\end{eqnarray}
where $\mathcal{F}_t$ is the filtration 
generated by the random walk up to time $t$.
Taking expectations on both sides of \eqref{telescope}, we get
$$\E_{(x,y)}[y_{{\tau}({A^c})}]-y
=\frac{1}{4n}\E_{(x,y)}\sum_{t=0}^{\infty}
\mathbf{1}({y_t=0})\mathbf{1}({\tau}({A^c})>t)=G_{A}(x,y)$$ 
and hence we are done.
\end{proof}
\begin{remark}\label{stopped height more}
Note that the above lemma implies 
for any $(x,y)\,\in\,\Cyl_n$, 
\begin{equation*}\label{biggerheight}
\E_{(x,y)}[y_{{\tau}({A^c})}]\ge y
\end{equation*}
since the Green's function is a non-negative quantity.
\end{remark}

Next we relate the Green's function to the solution of a variational problem. The results are well known and classical even though our setup is slightly different. Hence we choose to include the proofs for clarity. As defined in subsection \ref{secef} let $\vec{E}$ denote the set of directed edges of $\Cyl_n.$

For any function $F: \Cyl_n \rightarrow \mathbb{R}$
define the gradient $\nabla F : \vec{E} \rightarrow \mathbb{R}$ by
$$\nabla F (v,w) = F(w)-F(v)$$
and the discrete Laplacian $\Delta F : \Cyl_n \rightarrow \mathbb{R}$ by
\begin{equation}\label{lapldef}
\Delta F (v)=F(v) - \frac{1}{4}\sum_{w\sim v}F(w).
\end{equation}
Note that the graph $\Cyl_n$ is 4-regular.

Recall the definition of energy from subsection \ref{secef}. The next result is a standard summation-by-parts formula.
\begin{lemma}\label{byparts}For any function $F: \Cyl_n\rightarrow \mathbb{R}$
$$\mathcal{E}(\nabla F)=4\sum_{v\in \Cyl_n}F(v)\Delta F(v).$$
\end{lemma}
The proof follows by definition and expanding the terms.

For a subset $A \subset \Cyl_n$ recalling the definition of stopped Green's function let
\begin{equation}\label{generalflow}
f_A:=\nabla G_{A}.
\end{equation} 

Also recall the definition of divergence \eqref{divdef}.
\begin{lemma}\label{optflow1}For any $(x,y)\in A$
\begin{eqnarray*}
\div\bigl(f_A) (x,y)=\left\{\begin{array}{cl}
\frac{1}{n} & \mbox{if $y=0$,} \\
0 & \mbox{otherwise}. 
\end{array}
\right.
\end{eqnarray*}
\end{lemma}

\begin{proof}
 For any $v=(x,y)\in A$ by definition
\begin{equation}\label{energdiv}
\div(f_A)(v)=4\Delta G_{A}(v)=4G_{A}(v)-\sum_{w \sim v}G_{A}(w)=
\left\{\begin{array}{cl}
\frac{1}{n} & \mbox{if $y=0$,} \\
0 & \mbox{otherwise}. 
\end{array}
\right.
\end{equation} 
The last equality follows by the definition of $G_{A}$ in \eqref{gfa} by looking at the first step of the random walk started from $v.$
\end{proof}
We now prove that the random walk flow $f_{A}$ on a set $A$ is the flow with minimal energy.
\begin{lemma}\label{optflow2}
$$\mathcal{E}(f_A)=\inf_{f}\mathcal{E}(f)$$
where the infimum is taken over all flows 
from $\left (C_n \times \{0\}\right) \bigcap A$ to $A^c$
such that for $(x,y)\in A$
\begin{eqnarray*}
\div(f)(x,y)=\left\{\begin{array}{cl}
\frac{1}{n} & \mbox{if $y=0$,} \\
0 & \mbox{otherwise.} 
\end{array}
\right .
\end{eqnarray*}

\end{lemma}

\begin{proof}
The proof follows by standard arguments, see \cite[Theorem 9.10]{lpw}. 
We sketch the main steps. 
One begins by observing that the flow $f_{A}$ satisfies the cycle law, 
i.e.\ the sum of the flow along any cycle is $0$. 
To see this notice that for any cycle $$x_1,x_2,\ldots x_k=x_1$$ where $x_{i}'s\in \Cyl_n,$
$$\sum_{i=1}^{k-1}f_{A}(x_{i},x_{i+1})=\sum_{i=1}^{k-1}(G_{A}(x_{i+1})-G_{A}(x_{i}))=0.$$
The proof is then completed by first showing that the flow with the minimum energy must satisfy the cycle law, followed by showing that there is an unique flow satisfying the given divergence conditions and the cycle law.

\end{proof}

Now suppose $A\subset \Cyl_{n}\setminus \left (C_{n}\times \{1\}\right)$. Then
\begin{equation}\label{argument3}
\mathcal{E}({f}_{A}) = \mathcal{E}(\nabla {G}_{A})
=\sum_{v \in A}{G}_{A}(v)4\Delta {G}_{A}(v)=\frac{1}{n}\sum_{k \in C_n}{G}_{A}(k,0)=\frac{1}{n}\sum_{k \in C_n}H_{A}(k,0).
\end{equation}
The first equality is by definition. The second equality follow from Lemma \ref{byparts} and the fact that $G_{A}$ is $0$ outside $A$. The third equality is by \eqref{energdiv}. The last equality is by Lemma \ref{stopgre} since by hypothesis $A \cap (C_{n}\times \{1\})=\emptyset$. \\

\textit{Proof of Lemma \ref{heightgreen1}}. The proof now follows from \eqref{argument3} and Lemma \ref{optflow2}.
 \qed
\section{ Proof of Lemma \ref{lemm:azuma1}}\label{newapp}
We first prove $i.$ 
Looking at the process $\omega(t)$ started from $\omega(0)=\omega$ 
we see by \eqref{submart1} that 
the process $$Z_t=X_{t\wedge\tau(B)}-c[t\wedge\tau(B)]$$ with $X_0=g(\omega)$
is a submartingale with respect to the filtration $\mathcal{F}_t$.
Also by hypothesis $|Z_{t+1}-Z_{t}|\le 2A_2.$ 

Now by the standard Azuma-Hoeffding inequality for submartingales, 
for any time $t>0$ such that $a_2-a_1t<0$ we have
\begin{equation}\label{azumabound}
\mathbb{P}\left(Z_{t}-Z_0\le -(a_1t-a_2)\right)<e^{-\frac{{(a_1t-a_2)}^2}{4A_2^2t}}.
\end{equation}
Let $T$ be as in the hypothesis of the lemma. We observe that the event $\{X_{0} \ge A_1-a_2\} \cap\{ \tau(B)>T\}$ implies that
$$Z_{T}-Z_0\le -(a_1T-a_2).$$
This is because by hypothesis $Z_0=X_0\ge A_1-a_2$. Hence  on the event $\tau(B)>T$ 
$$Z_{T}-Z_0=X_T-a_1 T-X_0< a_2-a_1T$$ since $X_T\le A_1$ by \eqref{hyp111}.  
Thus by \eqref{azumabound} 
$$\mathbb{P}(\tau(B)>T\mid X_{0}>A_1-a_2)\le e^{-\frac{{(a_1T-a_2)}^2}{4A_2^2T}}.$$

To prove $ii.$ let $\omega_0=\omega(\tau(B^c)).$ By hypothesis $$x:=f(\omega_0)\ge A_1-a_4-A_2.$$
since by \eqref{hyp222} the process cannot jump by more than $A_2.$
Clearly it suffices to show  $$\mathbb{P}_{\omega_0}(\tau(B') \ge T') \ge 1-e^{-\frac{a_4^2 }{32A_2^2T}}+e^{-\frac{a_1^2 T}{32A_2^2}}.$$
Now  consider the submartingale $$W_t=X_{t\wedge\tau'\wedge \tau''}-a_1[{t\wedge\tau'\wedge \tau''}]$$
with $W_0=x$, where $\tau'=\tau(B)$ and $\tau''=\tau(B')$.  
We first claim that
\begin{equation}\label{firststep1} 
\mathbb{P}_{\omega_0}(\tau'\wedge \tau''>T)<e^{-\frac{c^2 T^2}{16C^2T}}.
\end{equation}
\noindent
To see this
%
notice that by the Azuma-Hoeffding inequality it follows that 
\begin{equation}\label{azuma1}
\mathbb{P}(W_{T}-W_{0}< -a_1 T/2 )<e^{-\frac{a_1^2 T^2}{16A_2^2T}}.
\end{equation} 
On the other hand the event $\tau'\wedge \tau''> T$ implies 
\begin{align*}
W_{T}& \le  A_1 -a_4-a_1T\\
W_{0} & \ge  A_1-a_4-A_2.
\end{align*}
Thus the event $\tau'\wedge \tau''> T$ implies $$W_{T}-W_{0}< -\frac{a_1 T}{2},$$ since by  hypothesis  $T> \frac{2A_2}{a_1}$.
\eqref{firststep1} now follows from \eqref{azuma1}.\\
 
Now on the event $\{\tau'\wedge \tau'' \le T\}\bigcap \{\tau''< \tau'\}$ ,
\begin{eqnarray*}
W_{T}& < & A_1-2a_4\\
W_{0}& \ge & A_1-a_4-A_2.
\end{eqnarray*}
Hence  $$\{\tau'\wedge \tau''  \le T\}\cap \{\tau''< \tau'\}\implies W_{T}-W_{0}\le -\frac{a_4}{2}$$ since by hypothesis $a_4> 2A_2$.
Thus by \eqref{azuma1}  we have 
$$\mathbb{P}\left(\{\tau'\wedge \tau''  \le T\} \bigcap \{\tau''\le \tau'\} \right)\le e^{-\frac{a_4^2 }{16A_2^2T}}.$$
Observe that $$\mathbb{P}(\tau'\le \tau'')\ge \mathbb{P}(\tau'\wedge \tau'' \le T)-\mathbb{P}\left(\{\tau'\wedge \tau''<T\} \bigcap \{\tau''\le \tau'\} \right).$$
This along with \eqref{firststep1} imply that
 $\tau''$ stochastically dominates a geometric variable with success probability at most $$e^{-\frac{a_4^2 }{16A_2^2T}}+e^{-\frac{a_1^2 T}{16A_2^2}}.$$ 
 Thus we are done.
\qed
\section*{Future directions and related models.}\label{concl1}

\paragraph{Fluctuations.}
This article establishes that competitive erosion on the cylinder forms a macroscopic interface quickly. A natural next step is to find the order of magnitude of its fluctuations. Theorem \ref{mainresult} only shows that the fluctuations are $o(n)$.

\paragraph{Randomly evolving interfaces.}
\begin{figure}
\centering
\begin{tabular}{ccc}
\includegraphics[scale=.5]{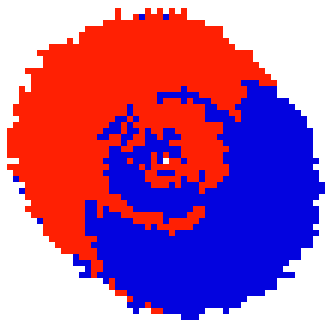} ~ & ~
\includegraphics[scale=.5]{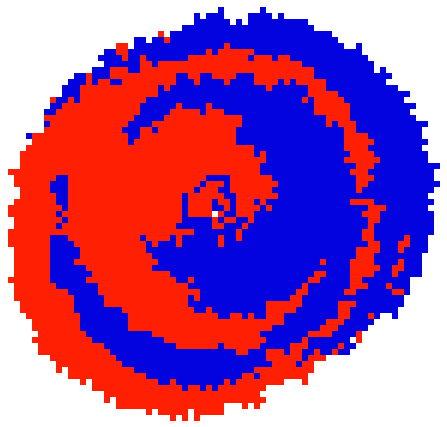} ~ & ~
\includegraphics[scale=.5]{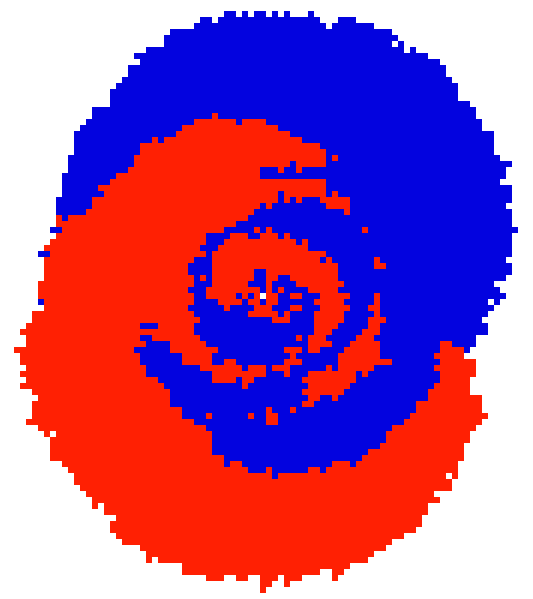} \\
$10^8$ & $10^9$ & $10^{10}$
\end{tabular}
\caption{A variant of competitive erosion on $\Z^2$ with $\mu_1 = \mu_2 = \delta_0$. 
From left to right, the set of all converted sites after $n= 10^8, 10^9, 10^{10}$ particles have been released. Most red particles convert blue sites and vice versa, so that a relatively small number of sites are converted.}
\label{f.spiral}
\end{figure}

Competitive erosion on the cylinder models a random interface fluctuating around a fixed line.
It can also model a moving interface if the measures $\mu_1$ and $\mu_2$ are allowed to depend on time. An interesting example is $\mu_i(t) = \delta_{Z_i(t)}$ where $Z_1$ and $Z_2$ are simple random walks independent of everything else in the process: that is, red and blue walkers are alternately released from a red source and blue source that themselves perform random walk on a slower time scale.  

Another model of a randomly evolving interface arises in the case 
of fixed but equal measures $\mu_1 = \mu_2$.  Figure~\ref{f.spiral} 
shows a variant of competitive erosion in the square grid $\Z^2$.
Initially all sites are colored white. 
Red and blue particles are alternately released from the origin.  
Each particle performs random walk until reaching 
a site in $\Z^2 -\{(0,0)\}$ colored differently from itself, 
and converts that site to its own color. 
Particles that return to the origin before converting a site are killed. 
One would not necessarily expect \emph{any} interface 
to emerge from this process, but simulations show 
surprisingly coherent red and blue territories.

\paragraph{Conformal invariance.}
 Our choice of the cylinder graph with uniform sources $\mu_i$ 
on the top and bottom is designed to make the function $g$ in the level set heuristic (see \eqref{e.generalg}) as simple as possible: 
$g(x,y) = 1-y$. A candidate Lyapunov function for more general graphs is
	\[ h(S) = \sum_{v \in S^c} g(v) \]
whose maximum over $S \subset V$ of cardinality $k$ 
is attained by the level set \eqref{e.levelset}.  

A case of particular interest is the following: Let $V = D \cap (\frac1n \Z^2)$ where $D$ is a bounded simply connected planar domain.  We take $\mu_i = \delta_{z_i}$ for points $z_1,z_2 \in D$ adjacent to $D^c$. 
As the edges of our graph we take the usual nearest-neighbor edges of $\frac1n \Z^2$ and delete every edge between $D$ and $D^c$. In the case that $D$ is the unit disk with $z_1 = 1$ and $z_2 = -1$, the level lines of $g$ are circular arcs meeting $\partial D$ at right angles.  The location of the interface for general $D$ can then be predicted by conformally mapping $D$ to the disk.
Extending the key Theorem~\ref{nonnegativedrift} to the above setup is a technical challenge we address in a subsequent paper.
\section*{Acknowledgments}
We thank Gerandy Brito and Matthew Junge for helpful comments.
The work was initiated when S.G. was an intern with the Theory Group at Microsoft Research, Redmond and a part of it was completed when L.L. and J.P were visiting. They thank the group for its hospitality.
\bibliographystyle{alpha}
\bibliography{GFF}

\end{document}